\documentclass[11pt,reqno,a4paper]{amsart}

\usepackage[utf8]{inputenc}
\usepackage{cancel}
\usepackage[margin=0.8in]{geometry}
\usepackage[usenames]{color}
\usepackage{amsmath,amsthm,pdfsync,verbatim,graphicx,epstopdf,enumerate,amssymb}
\usepackage[colorlinks=true]{hyperref}
\usepackage{float}
\usepackage{cancel,cite}
\usepackage[framemethod=tikz]{mdframed}
\usepackage{mathtools}
\mathtoolsset{showonlyrefs}

\hypersetup{allcolors=blue}
\numberwithin{equation}{section}

\newcommand{\lb}{\left(}

\newcommand{\rb}{\right)}
\newcommand{\PD}{\partial}

\newcommand{\Beq}{\begin{equation}}
	\newcommand{\Eeq}{\end{equation}}
\newcommand{\beq}{\begin{equation*}}
	\newcommand{\eeq}{\end{equation*}}
\newcommand{\bal}{\begin{align}}
	\newcommand{\eal}{\end{align}}

\newcommand{\bp}{\begin{prob}}
	\newcommand{\ep}{\end{prob}}
\newcommand{\bpr}{\begin{proof}}
	\newcommand{\epr}{\end{proof}}
\renewcommand{\o}{\omega}

\newcommand{\tblue}[1]{{\color{blue}{#1}}}

\newcommand{\bel}[1]{\begin{equation}\label{#1}}
	\newcommand{\ee}{\end{equation}}
\newcommand{\norm}[1]{|{#1}|}

\allowdisplaybreaks
\newtheorem{theorem}{Theorem}[section]

\newtheorem{lemma}[theorem]{Lemma}
\newtheorem{proposition}[theorem]{Proposition}

\theoremstyle{definition}

\newtheorem{remark}[theorem]{Remark}

\newcommand{\D}{\mathrm{d}}
\newcommand{\Lc}{\mathcal{L}}
\newcommand{\Rb}{\mathbb{R}}

\newcommand{\A}{\alpha}

\newcommand{\Bb}{\mathbb{B}}
\usepackage[usenames]{color}
\usepackage{amsmath,pdfsync,verbatim,graphicx,epstopdf,enumerate}

\newcommand{\wt}{\widetilde}

\newcommand{\Rc}{\mathcal{R}}

\newcommand{\Sb}{\mathbb{S}}
\newcommand{\Sn}{\mathbb{S}^{n-1}}

\renewcommand{\o}{\omega}
 \newcommand{\pcc}[1]{\begin{quotation}\textbf{\color{teal}Pradipta's comment:\
 		}{\color{teal}\textit{#1}}\end{quotation}}

\title[A constructive approach to range characterization of SMT odd dimensions]{A constructive approach to range characterization of spherical mean transform in odd dimensions} 

\begin{document}
\author[Chatterjee, Krishnan and Tushir]{Pradipta Chatterjee, Venkateswaran P.\ Krishnan and Abhilash Tushir}
\address{
	Tata Institute of Fundamental Research, Centre For Applicable Mathematics
    \endgraf Bangalore, Karnataka 560065, India
	\endgraf
    \it Email: \tt pradipta22@tifrbng.res.in, vkrishnan@tifrbng.res.in, abhilash2296@gmail.com 
}
\maketitle
\begin{abstract} This work focuses on range characterization of spherical mean transform (SMT) in odd-dimension Euclidean space. In a recent work, a simpler description characterizing the range of SMT in odd dimensions was derived. This involved symmetry relations involving certain linear ordinary differential operators acting on the coefficients of the spherical harmonics expansion of the function in the range. The sufficiency part of this characterization was shown using an existing range characterization result due to Agranovsky-Kuchment-Quinto. In the current work, we provide a direct and constructive proof showing the sufficiency of the aforementioned symmetry conditions and we obtain this without invoking any of the existing range characterization results. As an added advantage, given a function $g$ satisfying the range conditions, our approach provides an explicit construction of a function $f$  such that the SMT of $f$ is $g$. 
\end{abstract}
\section{Introduction and main results}

The spherical mean transform, commonly known as spherical Radon transform (SMT), maps a function to its integrals over spheres in $\mathbb{R}^n$. More precisely, given a 
function $f \in C_c^\infty (\Bb)$, where $\Bb$ is the unit ball in $\Rb^n$,  the \textit{spherical mean transform} $\Rc$ is defined as 
\begin{equation*}
    \Rc f (p, t) = \frac{1}{\o_n} \int\limits_{\Sn} f(p+t\theta) \, \D S(\theta), \quad (p,t)\in \Sb^{n-1}\times (0,\infty).
\end{equation*}
Here $\o_n$ denotes the surface area of $\Sn$ and $\D S$ denotes the surface measure on it. 
Due to the support restriction on $f$, $\Rc f (\cdot, t) = 0$ for $t \geq 2$. Thus we have $\Rc: C_c^\infty(\Bb) \to C_c^\infty (\Sn \times (0,2))$.

The SMT has a long and rich history,  naturally arising in the study of partial differential equations, particularly the wave and Euler--Poisson--Darboux equations; see \cite{CH_Book,John-book,Rhee}. Since then, the transform has played an important role in approximation theory; see \cite{agranovsky1996approximation,Eknaraynan2007,EKnarayanan2009} and functional analysis; see \cite{agranovsky1996injectivity,Rama2009,Rama2011}. In addition, it has become a fundamental tool in integral geometry with various applications in inverse problems, particularly in thermoacoustic and photoacoustic tomography; see \cite{Agranovsky:1999,AQ2, AQ3, Finch_2006,And,Denisjuk1999,Kuchment-Kunyansky-TAT,Kruger:1995}. A fundamental inverse problem in these applications is the reconstruction of the function from its spherical averages.  Due to its formally overdetermined nature, it is typically studied under suitable geometric constraints, such as restricting the centers of the spheres to a prescribed hypersurface or limiting the range of admissible radii; see \cite{Finch-P-R,Palamodov2016,ambartsoumian2015inversion,Ambartsoumian-Zarrad-Lewis,aramyan2020recovering,Finch-Haltmeir-Rakesh_even-inversion,CKT,K,R,Salman:2014,ambartsoumian2014exterior,nguyen2009family,norton1980reconstruction}. The problem of inversion from restricted data has also been extensively investigated; see \cite{Ambartsoumian2018,Ambartsoumian2015,CKT,Salman,Kuchment-Kunyansky-TAT,DK}. A distinctive feature of the inversion formulas for SMT is its dependence on the  ambient dimension. In odd dimensions, reconstruction involves local inversion formulas, whereas in even dimensions the inversion formula is  non-local, typically involving singular or integral operators. The characterization of injectivity sets for the spherical mean transform in various classes of functions is a fundamental problem as well that has attracted considerable attention; see \cite{agranovsky1996injectivity,ref:AmbKuch,Agronovasky96,Clark1999}.

In the study of the spherical Radon transform, besides investigating inversion formulas and injectivity sets, it is desirable to have a description of the range of the transform. Apart from theoretical interest, a complete description of the range is valuable in applications since the measured data can be noisy or have missing parts due to restrictions in measurements or having inaccessible boundaries from which measurements are not possible. In such instances, the knowledge of the range of the transform can help suppress the noise or fill in the missing data \cite{Natterer_book,Anastasio2001FBP,SKPatch_2004,Nat83}. 

The range characterization of the spherical mean transform has been extensively studied over the past few decades. The current work is a follow-up of the article \cite{AAKN1}, where a comprehensive survey of the literature on the range characterization of the spherical mean transform is already provided; we refer the reader to this article for the relevant references. The literature contains several range characterization results for the spherical mean transform; see \cite{Agranovsky-Finch-Kuchment-range,Finch_2006,Agranovsky-Kuchment-Quinto,AN,ref:AmbKuch-range,Linh,KK,AAKN1}. Among these, the recent characterization established in \cite{AAKN1} serves as the starting point of the current work because of its simple and transparent formulation. This range characterization is particularly well suited for constructive proofs, which is the main focus of this paper.

We summarize the range characterization results from \cite{AAKN1} now. 
Given a function $g\in C_c^{\infty}(\Sb^{n-1}\times (0,2))$, let us expand it in spherical harmonics: 

\begin{equation*}
     g(\theta,t)=\sum_{q=0}^{\infty}\sum_{s=1}^{d_{q}} g_{q,s}(t)Y_{q,s}(\theta),
\end{equation*}
where $g_{q,s}\in C_{c}^{\infty}((0,2))$,\begin{equation}\label{dqandfqs}
 \quad d_{q}=\frac{(2q+n-2)(n+q-3)!}{q!(n-2)!},\quad d_{0}=1,\quad \text{and}\quad g_{q,s}(t)=\int_{\mathbb{S}^{n-1}}g(t\theta)\overline{Y}_{q,s}(\theta)\D \theta.
\end{equation} 

\begin{theorem}[Range characterization for SMT of general functions]\label{generalrange}
    \cite{AAKN1} Let $\Bb$ denote the unit ball in $\Rb^n$ for an odd
$n\geq 3$, and $k := \frac{n-3}{2}$. A function $g\in C_c^{\infty}(\Sb^{n-1}\times (0, 2))$ is representable as $g= \mathcal{R}f$ for $f\in C_c^{\infty}(\Bb)$
if and only if for each $(q, s), q\geq 0, 1 \leq s \leq d_{q}, h_{q,s}(t) = t^{n-2}g_{q,s}(t)$ satisfies the following two consistency 
conditions:
\begin{itemize}
\item there is a function $\phi_{q,s}\in C_c^{\infty}((0,2))$ such that
\begin{equation}\label{Gencond1}
    h_{q,s}(t) =D^q\phi_{q,s}(t), \quad D=\frac{1}{t}\frac{\D}{\D t},
\end{equation}
\item the function $\phi_{q,s}(t)$ satisfies
\begin{equation}\label{Gencond2}
    \left[\mathcal{L}_{q+k} \phi_{q,s}\right](1-t)=\left[\mathcal{L}_{q+k}\phi_{q,s}\right](1+t),
\end{equation}
where $\mathcal{L}_k$ is the linear differential operator of order $k$:
\begin{equation}\label{lkh:operator}
    \mathcal{L}_k=\sum_{l=0}^k \frac{(k+l)!}{(k-l)!l!2^l}(1-t)^{k-l} D^{k-l},
\end{equation}
and $\left[\mathcal{L}_k h\right](\cdot)$ denotes evaluation of the function $\mathcal{L}_k h$ at the given point.
\end{itemize}
\end{theorem}
If we assume that $g$ is independent of the angular coordinate, then the above theorem simplifies to the following result: 

\begin{theorem}[Range characterization for SMT of radial functions] \cite{AAKN1}\label{RC:Radial} Let $\mathbb{B}$ denote the unit ball in $\mathbb{R}^n$ for an odd $n \geq 3$, and $k:=(n-3) / 2$. A function $g \in C_c^{\infty}((0,2))$ is representable as $g=\mathcal{R} f$ for a radial function $f \in C_c^{\infty}(\mathbb{B})$ if and only if $h(t):=t^{n-2} g(t)$ satisfies
\begin{equation}\label{rangeee}
    \left[\mathcal{L}_k h\right](1-t)=\left[\mathcal{L}_k h\right](1+t), \quad \text { for all } t \in[0,1]
\end{equation}
where $\mathcal{L}_k$ is the operator in \eqref{lkh:operator}.

\end{theorem}
In \cite{AAKN1}, the authors proved the necessity of two conditions in Theorem \ref{generalrange} by an application of Funk-Hecke theorem \cite[Theorem 3]{Seeley} and for the proof of sufficiency, they relied on the following range characterization result of Agranovsky-Kuchment-Quinto.

\begin{theorem}\cite[Theorem~11]{Agranovsky-Kuchment-Quinto}\label{range-AKQ}
    Let $n > 1$ be an odd integer. A function $g \in C_c^\infty(\Sn \times (0,2))$ is representable as $\Rc f$ for some $f \in C_c^\infty(\Bb)$ if and only if for any $m$, the $m^{\mathrm{th}}$ order spherical harmonic term $\widehat{g}_m (p,\lambda)$ of $\widehat{g}(p,\lambda)$ vanishes at non-zero zeros of the Bessel function $J_{m+n/2-1}(\lambda)$, where
    \begin{align*}
        \widehat{g}(p,\lambda) =\int\limits_0^\infty g(p,t) j_{\frac{n-2}{2}}(\lambda t) t^{n-1} \, \D t
    \end{align*}
    is the Hankel transform of $g$. 
\end{theorem}

Our first motivation for this paper is to provide a direct and constructive proof of sufficiency of the conditions given in \cite{AAKN1} without the use of any of the existing range characterization results; in particular the aforementioned result of Agranovsky-Kuchment-Quinto. 
Additionally, we note that the characterization result in Theorem \ref{range-AKQ} gives the existence of a function $f\in C_c^{\infty}(\Bb)$ such that $\Rc f =g$, but does not, however, give an explicit representation for such an $f$. This leads us to our second motivation.  We are interested in an explicit construction of an $f\in C_c^{\infty}(\Bb)$ such that $\Rc f =g$, where $h(p,t)=\dfrac{\omega_{2k+3}t^{2k+1}}{\omega_{2k+2}}g(p,t)$ satisfies conditions \eqref{Gencond1} and \eqref{Gencond2}.
More precisely, our main result states that if a function $g\in C_c^{\infty}(\Sb^{n-1}\times (0,2))$ satisfies the two conditions of Theorem \ref{generalrange}, then there exists an \emph{explicit} function $f\in C_c^{\infty}(\Bb)$ such that $\Rc f =g$. We again note that we do not rely on Theorem \ref{range-AKQ} in the construction of this $f$. It provides an alternate and reconstructive proof of the sufficiency part of the range characterization result in \cite{AAKN1}. 

The proof for the range characterization in this article can be viewed in the same spirit as the one for the classical Radon transform, in the sense that the range characterization result for the Radon transform in Schwartz spaces provides an explicit candidate through the use of Fourier slice theorem. The integral moment conditions are required to show that this candidate function is smooth across the origin. Similarly, in our set-up, we give a candidate function whose SMT is the given function in the range; see \eqref{Eq f expression}. We note that this function is well-defined in $\Bb\setminus \{0\}$. The conditions in Theorem \ref{generalrange} are used to show that this function extends smoothly across the origin. 

Our approach has a few other advantages that we would like to highlight. 
\begin{itemize} 
\item From the inversion formula we obtained in this work (see \eqref{Eq f expression}), the support of the reconstructed function is immediately evident. This is in contrast with most existing inversion formulas, which are formulated as integral representations and in such cases, the support of the reconstructed function  is not immediate. 
\item We also observe from our inversion formula that to determine the value of $f(x)$ at a point \(x=t\theta\), it suffices to know the given data and its derivatives on the cylinder  $\Sb^{n-1}\times[1-t,1+t]$.
\item A final feature is that the inversion formula involves only $k+1$ derivatives of the data; see \eqref{General f expression}, whereas the existing formulas in the literature require $2k+2$ derivatives on the data; see \cite{R, Finch-P-R}. 

\end{itemize} 

This article is organized as follows.
Section \ref{sec:mainresults} presents the main results of this article. In Section \ref{sec:prepresults}, we establish the preliminary results that are essential for the proofs of the main theorems. In Section \ref{equiv:repoff} we present equivalent representations for our candidate function. Sections \ref{Mthm:step2} and \ref{sec:genfunction}  contain the proofs of the main results in the case of functions independent of the angular variable and the general case respectively.
\subsection{Main results}\label{sec:mainresults}


We now state the first main result of this article, which provides an explicit reconstruction formula for radial functions from the given function $g$ (independent of the angular variable $p$) satisfying the  condition \eqref{rangeee}:

\begin{theorem}[Explicit reconstruction for radial functions] \label{Mainthm1} Let $\mathbb{B}$ denote the unit ball in $\mathbb{R}^{n}$ 
for an odd integer $n\geq 3$, and $k:=\frac{n-3}{2}$.
Let $g\in C_c^{\infty}((0,2))$ satisfy the range condition \eqref{rangeee}.
Then the unique $f\in C_c^{\infty}(\Bb)$ such that $\mathcal{R}f=g$ is given by: 
\begin{itemize}
\item for $k=0$: 
\[
f(t)=\frac{h'(1-t)-h'(1+t)}{2t},
\]
\item for $k\geq 1$: 
\[
f(t)=\frac{1}{2^{3k+1}k!}\sum_{j=0}^{k-1}(-1)^{k-j-1}\binom{k-1}{j}D^{k-j}\left(\frac{[D^{j+1}h_{k}](1+t)-[D^{j+1}h_{k}](1-t)}{t}\right),
\]
where $h_{k}(t)=\dfrac{4^{k}\omega_{2k+3}t^{2k+1}}{\omega_{2k+2}}g(t)$. 
\end{itemize} 
\end{theorem}

We now consider the general case. The spherical harmonic expansion of $g\in C_{c}^{\infty}(\mathbb{B})$ is given by
\begin{equation}
        g(x)=\sum_{q=0}^{\infty}\sum_{s=1}^{d_{q}} g_{q,s}(|x|)Y_{q,s}\left(\frac{x}{|x|}\right),
\end{equation}
where
\begin{equation}\label{dqandfqs}
 \quad d_{q}=\frac{(2q+n-2)(n+q-3)!}{q!(n-2)!},\quad d_{0}=1,\quad \text{and}\quad g_{q,s}(r)=\int_{\mathbb{S}^{n-1}}g(r\theta)\overline{Y}_{q,s}(\theta)\D \theta.
\end{equation}
 
 \begin{theorem}[Explicit reconstruction for general functions]
     Let $\mathbb{B}$ denote the unit ball in $\mathbb{R}^{n}$ 
for an odd integer $n\geq 3$, and $k:=\frac{n-3}{2}$. Let $g\in C_{c}^{\infty}(\mathbb{S}^{n-1}\times (0,2))$ and the spherical harmonics expansion of $g$ is given by
\begin{equation*}
     g(\theta,t)=\sum_{q=0}^{\infty}\sum_{s=1}^{d_{q}} g_{q,s}(t)Y_{q,s}(\theta),
\end{equation*}
where $g_{q,s}\in C_{c}^{\infty}((0,2))$. If we denote  $ h_{q,s}(t):=\dfrac{\omega_{2k+3}t^{2k+1}}{\omega_{2k+2}}g_{q, s}(t)$ satisfying the range condition, then the reconstruction for $f\in C_{c}^{\infty}(\mathbb{B})$ such that $\mathcal{R}f=g$ is given by
\begin{align*}
    f(x)=\sum_{q=0}^{\infty}\sum_{s=1}^{d_{q}} f_{q,s}(t)Y_{q,s}(\theta), \quad x=t\theta,
\end{align*}
where $f_{q,s}(t)$ is given by:
\begin{itemize}
\item for $k=0$ and $q=0$:
\begin{equation}\label{General f expression k=0, q=0 case}
    f_{0,1}(t)=\lb\frac{h_{0,1}^{\prime}(1-t)-h_{0,1}^{\prime}(1+t)}{2t}\rb,
\end{equation}
\item for $k=0$ and $q\geq 1$:
\Beq\label{General f expression k=0 case}
\begin{aligned}
    {f}_{q,s}(t)&=\frac{(-1)^{q-1}(t^2-1)}{t^{q+1}2^{3q-2}(q-1)!}\int_{1-t}^{1+t}\frac{d}{du}[D^{q}h_{q,s}](u)(Q(t,u))^{q-1}\D u\\
    &=\frac{(t^2-1)}{t^{q+1}2^{3q-2}(q-1)!}\int_{1-t}^{1+t}\frac{d}{du}[Dh_{q,s}](u)D^{q-1}\lb (Q(t,u))^{q-1}\rb \, \D u,
\end{aligned}
\Eeq
\item for $k\geq 1$ and $q\geq 0$:
\Beq\label{General f expression}
\begin{aligned}
    {f}_{q,s}(t)&=\frac{(-1)^{k+q-1}(t^2-1)}{t^{2k+q+1}2^{4k+3q-2}k!(k+q-1)!}\int_{1-t}^{1+t}\frac{d}{du}[D^{2k+q}h_{q,s}](u)(Q(t,u))^{k+q-1}\D u\\
    &=\frac{(t^2-1)}{t^{2k+q+1}2^{4k+3q-2}k!(k+q-1)!}\int_{1-t}^{1+t}\frac{d}{du}[D^{k+1}h_{q,s}](u)D^{k+q-1}\lb (Q(t,u))^{k+q-1}\rb \, \D u\\
&= \frac{t^2 - 1}{t^{2k+q+1} \, 2^{4k+3q-2} \, k! \, (k+q-1)!} 
\Bigg{\{} \left. \left[ D^{k+1} h_{q,s} \right](u) \, \left( D^{k+q-1} \left[ (Q(t,u))^{k+q-1} \right] \right) \right|_{u=1-t}^{u=1+t} \\
&\quad - \int_{1-t}^{1+t} \frac{\mathrm{d}}{\mathrm{d}u} \left[ D^k h_{q,s} \right](u) \, D^{k+q} \left[ (Q(t,u))^{k+q-1} \right] \mathrm{d}u \Bigg{\}}.
\end{aligned}
\Eeq
\end{itemize}
 \end{theorem}

\section{Preliminaries}\label{sec:prepresults}
In this section, we recall certain known results that will be employed throughout the paper.

 We begin by recalling a few  estimates related to spherical harmonics. The following estimates are well-known:
 \begin{enumerate}
 \item\label{fact1} $d_{q}=\mathcal{O}(q^{n-2})$ for $q$ sufficiently large; see \cite[equation (2.12)]{Atk:book};
     \item\label{fact2} $\max\left\{|Y_{q,s}(\theta)|,\theta\in \mathbb{S}^{n-1},~~1\leq s\leq d_{q}\right\}\leq \left(\frac{d_{q}}{|\mathbb{S}^{n-1}|}\right)^{\frac{1}{2}}\lesssim \sqrt{d_{q}}$; see \cite[equation (2.36)]{Atk:book}; and
     \item\label{fact3} any derivative of $g_{q,s}(t)$ for $t\in (0,2)$ decays faster than any polynomial. That is, for every $N,m\in\mathbb{N}_{0}$ there exists $C_{N,m}>0$ such that  $|\partial_{t}^{m}g_{q,s}(t)|\leq \frac{C_{N,m}}{q^{N}}$. 
     This can be seen as follows: 
     \begin{align*}
        |\partial_{t}^{m}g_{q,s}(t) |=\left|{\int_{\Sb^{n-1}}}\PD_t^m g(t,\theta)Y_{q,s}(\theta)\right|=&\left|\frac{1}{(q(q+n-2))^N}{\int_{\Sb^{n-1}}}\PD_t^m g(t,\theta)\Delta_{\Sb^{n-1}}^NY_{q,s}(\theta)\right|\\
        =&\left|\frac{1}{(q(q+n-2))^N}.{\int_{\Sb^{n-1}}}\Delta_{\Sb^{n-1}}^N\PD_t^m g(t,\theta)Y_{q,s}(\theta)\right|
        \leq \frac{C_{N,m}}{q^N}.
     \end{align*}
 \end{enumerate}
 For  radial functions, using Funk-Hecke theorem \cite[Theorem 3]{Seeley}, we have
\begin{align*}
\Rc f(p,t)=\frac{C_{n}}{t^{n-2}}\int\limits_{|1-t|}^{1} uf(u) [Q(t,u)]^{\frac{n-3}{2}} \D u,
\end{align*}
where
$$C_{n}=\frac{\omega_{n-1}}{\omega_{n}2^{n-3}}\quad\text{and}\quad Q(t,u)=((1+t)^2-u^2)(u^2-(1-t)^2).$$ 

Denoting $h(t)=\dfrac{t^{n-2}}{C_{n}}\mathcal{R}f(t)=\dfrac{t^{n-2}}{C_{n}}g(t)$ (note that since $f$ is radial, $\mathcal{R}f$ is independent of $p$), we get
\begin{equation*}
    h(t)=\int\limits_{|1-t|}^{1} uf(u) [Q(t,u)]^{\frac{n-3}{2}} \D u.
\end{equation*}
If we set $k:=\dfrac{n-3}{2}$ and using the notation $h_{k}(t)$ instead of $h(t)$ to emphasize  the dependence on $k$,  we have
\begin{equation}\label{hkdef}
    h_{k}(t)=\int\limits_{|1-t|}^{1} uf(u) [Q(t,u)]^{k} \D u.
\end{equation}

We will require the following result established in \cite{CKT,chatterjee2026unifiedrangecharacterizationspherical}: 
\begin{theorem}[Inversion algorithm for radial functions with partial radial data]\label{ODE:inversion} Let $\epsilon>0$, $n\geq 3$ be an odd integer, and $k:=\frac{n-3}{2}$. Also assume that $f\in C_{c}^{\infty}(\mathbb{B})$ and  $$h_{k}(t)=\dfrac{4^{k}\omega_{2k+3}t^{2k+1}}{\omega_{2k+2}}\mathcal{R}f(p,t),~~\forall~(p,t)\in \mathbb{S}^{n-1}\times (0,1),$$   then $f$ can be recovered in the annular region $\mathbb{B}(\epsilon,1)$  by solving
\begin{equation}\label{ODE:op}
			\left[\frac{\D}{\D t}D^{2k}h_{k}\right](t)
			=\frac{(-1)^{k}k!2^{3k}}{t^{2k}}   \left[\mathcal{L}_{k}(t^{2k+1}f)\right](1-t),
		\end{equation}
        with initial conditions $f^{(i)}(1-\epsilon^{\prime})=0$ for $0\leq i\leq k-1$, where $\epsilon^{\prime}>0$ is small enough such that $1-\epsilon^{\prime}$ lies outside the support of $f$.
\end{theorem}
\begin{remark}\label{rmk:t02}
    In the above theorem, if we instead utilize the data $h_{k}(t)$ for $t\in (1,2)$, then the ODE \eqref{ODE:op} is replaced by
    \begin{equation}\label{ODE:op2}
			\left[\frac{\D}{\D t}D^{2k}h_{k}\right](t)
			=\frac{(-1)^{k+1}k!2^{3k}}{t^{2k}}   \left[\widetilde{\mathcal{L}}_{k}(t^{2k+1}f)\right](t-1),
		\end{equation}
        where
        \begin{align*}
[\widetilde{\mathcal{L}}_kh_{k}](t)=\sum_{l=0}^{k}\frac{(k+l)!}{(k-l)!l!2^l}(1+t)^{k-l}[D^{k-l}h_{k}](t).
\end{align*}
\end{remark}
The following proposition follows easily from \cite[Theorem 2.1]{AAKN2}.
\begin{proposition}\label{Lem:Lktointegral}
For $\psi\in C_c^{\infty}((0,2))$ we have,
\begin{align*}
    \left[\mathcal{L}_{k}\psi\right](1+t)-\left[\mathcal{L}_{k}\psi\right](1-t)=\frac{(-1)^{k}}{2^{3k}k!}\int_{1-t}^{1+t}\left[\frac{d}{du}D^{2k}\psi\right](u)(Q(t,u))^{k}\D u,
\end{align*}    
where $\mathcal{L}_k$ is defined as in \eqref{lkh:operator}.    
\end{proposition}

\begin{proposition}\label{fder}\cite[Lemma 3.3]{CKT} For any $r\in\mathbb{N}$ and $\psi\in C^{r}(\mathbb{R})$, the following relation holds:
	\begin{equation}\label{dkf} 
		[D^{r}\psi](t)=\sum_{j=1}^{r}(-1)^{r-j}\frac{(2r-1-j)!}{2^{r-j}(r-j)!(j-1)!t^{2r-j}}\psi^{(j)}(t)=\sum_{j=0}^{r}B(j,r)t^{j-2r}f^{(j)}(t),
	\end{equation}
    where the coefficient $B(j,r):=(-1)^{r-j}\frac{(r-j)!}{2^{r-j}}\binom{r-1}{j-1}\binom{2r-1-j}{r-1}$.  
\end{proposition}
\section{Equivalent representations of the reconstruction function} 
In this section, we present our candidate reconstruction function.
For $n=3$ i.e., $k=0$, our candidate  is given by
\[f(t)=\frac{h_{0}^{\prime}(1-t)-h_{0}^{\prime}(1+t)}{2t}.\]
For $n\geq 5$, we also obtain an equivalent integral representation  of the reconstruction candidate, summarized in the following proposition:

\begin{proposition}\label{equiv:repoff}
Let $n\geq 5$ be odd, $k:=\frac{n-3}{2}$, and $h_{k}(t)\in C_c^{\infty}((0,2))$ satisfy condition \eqref{rangeee}. For $0<t<1$,  define a function $f$ by  
 \Beq\label{Eq f expression}
f(t)=\frac{1}{2^{3k+1}k!}\sum_{j=0}^{k-1}(-1)^{k-j-1}\binom{k-1}{j}D^{k-j}\left(\frac{[D^{j+1}h_{k}](1+t)-[D^{j+1}h_{k}](1-t)}{t}\right).
\Eeq Then $f$ has the following equivalent representations: 
   \begin{align}
      f(t)&=\frac{(t^2-1)}{2^{3k+1}k!t^{2k+1}}\Big{\{}\left[\mathcal{L}_{k-1}D^2h_k\right](1+t)-\left[\mathcal{L}_{k-1}D^2h_k\right](1-t)\Big{\}}\label{candidate2:new}\\
               &=\frac{(-1)^{k-1}(t^2-1)}{2^{6k-2}(k-1)!k!t^{2k+1}}\int_{1-t}^{1+t}\left[\frac{d}{du}D^{2k}h_k\right](u)(Q(t,u))^{k-1}\D u \label{candidate3:new}.
           \end{align}
\end{proposition}
\begin{remark}
    In our proofs, sometimes it is more convenient to use \eqref{candidate2:new} or \eqref{candidate3:new} rather than \eqref{Eq f expression}. This is the main reason for Proposition \ref{equiv:repoff}.
\end{remark}
\bpr 
In view of the reversibility of each step, it suffices to establish the result in one direction. 
Consider
\begin{align}
           f(t)=\frac{1}{2^{3k+1}k!}\sum_{j=0}^{k-1}(-1)^{k-j-1}\binom{k-1}{j}D^{k-j}\left(\frac{[D^{j+1}h_{k}](1+t)-[D^{j+1}h_{k}](1-t)}{t}\right).
           \end{align}
           By applying Leibniz’s rule involving the derivative $D^{k-j}$, we get 
   \begin{align}
     f(t)&
       =\underbrace{\frac{1}{2^{3k+1}k!}\sum_{j=0}^{k-1}\sum_{p=0}^{k-j}(-1)^{p+1}\binom{k-1}{j}\binom{k-j}{p}\frac{(2k-2j-2p)!}{2^{k-j-p}(k-j-p)!t^{2k-2j-2p+1}}D^{p}\{[D^{j+1}h_{k}](1+t)\}}_{f_{+}(t)}\nonumber\\&-\underbrace{\frac{1}{2^{3k+1}k!}\sum_{j=0}^{k-1}\sum_{p=0}^{k-j}(-1)^{p+1}\binom{k-1}{j}\binom{k-j}{p}\frac{(2k-2j-2p)!}{2^{k-j-p}(k-j-p)!t^{2k-2j-2p+1}}D^{p}\{[D^{j+1}h_{k}](1-t)\}}_{f_{-}(t)}\nonumber\\
       =&f_{+}(t)-f_{-}(t).\label{fplus}
\end{align}
The functions $f_{\pm}$ can be simplified as 
\begin{align*} 
    f_{\pm}(t)=\frac{1}{2^{3k+1}k!t^{2k+1}}M^{\pm}(t),
\end{align*}
where $M^{+}(t)$ and $M^{-}(t)$ are given by
    \begin{align}
 M^{+}(t)&=  \sum_{i=0}^{k}\sum_{s=i}^{i+1}\frac{(-1)^{i+1}(k-i)!}{2^{k-i}}\binom{k-s+i}{k-i}\binom{2k-s}{k-i}\binom{1}{s-i}t^{s}[D^{i+1}h_{k}](1+t),\quad\text{and}\\
  M^{-}(t)&= \sum_{i=0}^{k}\sum_{s=i}^{i+1}\frac{(-1)^{s+i+1}(k-i)!}{2^{k-i}}\binom{k-s+i}{k-i}\binom{2k-s}{k-i}\binom{1}{s-i}t^{s}[D^{i+1}h_{k}](1-t).
\end{align}
We show this in Lemma \ref{fpm:exp} below. 

Using the condition in \eqref{rangeee} and \eqref{fplus}, we can write 
\begin{equation}
    2^{3k+1}k!t^{2k+1}f(t)=2^{3k+1}k!t^{2k+1}f_{+}(t)+D(\left[\mathcal{L}_{k}h_{k}\right](1+t))-2^{3k+1}k!t^{2k+1}f_{-}(t)-D(\left[\mathcal{L}_{k}h_{k}\right](1-t)).
\end{equation}
Our aim is to establish the following claim:
\begin{enumerate}
    \item \Beq\label{claim1} 2^{3k+1}k!t^{2k+1}f_{+}(t)+D(\left[\mathcal{L}_{k}h_{k}\right](1+t))=(t^2-1)\left[\mathcal{L}_{k-1}D^2h_k\right](1+t)\Eeq and \item \Beq \label{claim2} 2^{3k+1}k!t^{2k+1}f_{-}(t)+D(\left[\mathcal{L}_{k}h_{k}\right](1-t))=(t^2-1)\left[\mathcal{L}_{k-1}D^2h_k\right](1-t).\Eeq
\end{enumerate}
Upon establishing the above claims, we get the required expression \eqref{candidate2:new}. An application of Proposition \ref{Lem:Lktointegral} then yields \eqref{candidate3:new}. 
We begin to prove the first claim. We have 
\begin{multline*}
       D([\mathcal{L}_k h_{k}](1+t) )=\sum_{l=0}^{k-1}\frac{(k+l)!}{(k-l-1)!l!2^l}(-1)^{k-l}t^{k-l-2}[D^{k-l}h_{k}](1+t)+\\\sum_{l=0}^k\frac{(k+l)!}{(k-l)!l!2^l}(-1)^{k-l}t^{k-l-1}(1+t)[D^{k-l+1}h_{k}](1+t).
\end{multline*}
Changing the variable $l \to k-1-l$ in first summation and $l\to k-l$ in second summation, we have
\begin{multline*}
    D([\mathcal{L}_k h_{k}](1+t) )=\sum_{l=0}^{k-1}\frac{(2k-l-1)!}{(k-l-1)!l!2^{k-l-1}}(-1)^{l+1}t^{l-1}[D^{l+1}h_{k}](1+t)+\\\sum_{l=0}^{k}\frac{(2k-l)!}{(k-l)!l!2^{k-l}}(-1)^{l}t^{l-1}(1+t)[D^{l+1}h_{k}](1+t).
\end{multline*}
On simplification, we get
\begin{align}\label{15}
    D([\mathcal{L}_k h_{k}](1+t) )
    =&\sum_{l=0}^{k}\frac{(-1)^l}{(k-l)!l!2^{k-l}}\left[{(2k-l-1)!l}t^{l-1}+(2k-l)!t^{l}\right][D^{l+1}h_{k}](1+t).
\end{align}
The expression of $f_{+}(t)$ from Lemma \ref{fpm:exp} in terms of factorials is
 \begin{align}
       2^{3k+1}k!t^{2k+1}f_{+}(t)
       &=\sum_{l=0}^{k}\frac{(-1)^{l+1}}{2^{k-l}(k-l)!l!}\left[(2k-l)!t^l+l(2k-l-1)!t^{l+1}\right][D^{l+1}h_{k}](1+t)\label{16}.
    \end{align}
Adding \eqref{15} and \eqref{16}, we have
\begin{align*}
    2^{3k+1}k!t^{2k+1}f_{+}(t)+D([\mathcal{L}_k h_{k}](1+t) )&=(t^2-1)\sum_{l=1}^{k}\frac{(-1)^{l+1}(2k-l-1)!}{2^{k-l}(k-l)!(l-1)!}t^{l-1}[D^{l+1}h_{k}](1+t)\\
    &=(t^2-1)\sum_{l=0}^{k-1}\frac{(-1)^{k-1-l}(k-1+l)!}{2^{l}(k-1-l)!l!}t^{k-1-l}[D^{k-1-l+2}h_{k}](1+t)\\
     &=(t^2-1)\left[\mathcal{L}_{k-1}D^2h\right](1+t),
\end{align*}
where the last but one step follows by
changing the variable $l \to  k-l$. This proves the first claim. A similar proof gives the second claim. 
This completes the proof of Proposition \ref{equiv:repoff}.
\epr 

\begin{lemma}\label{fpm:exp}
$f_{\pm}$ in \eqref{fplus} simplifies to
\begin{align*} 
    f_{\pm}(t)=\frac{1}{2^{3k+1}k!t^{2k+1}}M^{\pm}(t),
\end{align*}
where $M^{+}(t)$ and $M^{-}(t)$ are given by
    \begin{align}
 M^{+}(t)&=  \sum_{i=0}^{k}\sum_{s=i}^{i+1}\frac{(-1)^{i+1}(k-i)!}{2^{k-i}}\binom{k-s+i}{k-i}\binom{2k-s}{k-i}\binom{1}{s-i}t^{s}[D^{i+1}h_{k}](1+t),\quad\text{and}\\
  M^{-}(t)&= \sum_{i=0}^{k}\sum_{s=i}^{i+1}\frac{(-1)^{s+i+1}(k-i)!}{2^{k-i}}\binom{k-s+i}{k-i}\binom{2k-s}{k-i}\binom{1}{s-i}t^{s}[D^{i+1}h_{k}](1-t).
\end{align}
\end{lemma}
 For brevity, we present
the derivation only for $f_{+}(t)$, since the formula for $f_{-}(t)$ follows analogously.
\begin{proof}
Substituting the expression for $D^{p}\{[D^{j+1}h_{k}](1+t)\}$ given by \eqref{Dmr:exp}  in Section \ref{Appendix} into $f_{+}(t)$ above, we obtain
     \begin{multline*}
     f_{+}(t)=\frac{1}{2^{3k+1}k!}\sum_{j=0}^{k-1}\sum_{p=0}^{k-j}(-1)^{p+1}\binom{k-1}{j}\binom{k-j}{p}\frac{(2k-2j-2p)!}{2^{k-j-p}(k-j-p)!t^{2k-2j-2p+1}}\times\\\left[\sum_{i=0}^{p}\sum_{s=i}^{2i}\left[\frac{p!(-1)^{p-i}}{i!2^{p-i}}\binom{i}{s-i}\left[\binom{2p-s}{p-i}-2\binom{2p-s-1}{p-i-1}\right]\right]t^{s-2p}[D^{j+1+i}h_{k}](1+t)\right].
\end{multline*}
Simplifying, we have
\begin{multline*}
2^{3k+1}k!t^{2k+1}f_{+}(t)=\sum_{j=0}^{k-1}\sum_{p=0}^{k-j}\sum_{i=0}^{p}\sum_{s=i}^{2i}\frac{(-1)^{i+1}}{2^{k-j-i}}\binom{k-1}{j}\binom{k-j}{p}\binom{i}{s-i}\frac{(2k-2j-2p)!p!}{(k-j-p)!i!}\times\\\left[\binom{2p-s}{p-i}-2\binom{2p-s-1}{p-i-1}\right] t^{s+2j}[D^{j+1+i}h_{k}](1+t).
\end{multline*}
Upon making the substitutions $i \to i-j$ and $s \to s-2j$, and simplifying the resulting expression using $\binom{a}{b}\binom{b}{c}=\binom{a}{c}\binom{a-c}{a-b}$, we obtain
\begin{multline}\label{ftilde:4sum}
2^{3k+1}k!t^{2k+1}f_{+}(t)=\sum_{j=0}^{k-1}\sum_{p=0}^{k-j}\sum_{i=j}^{p+j}\sum_{s=i+j}^{2i}\frac{(-1)^{i+j+1}(k-i)!}{2^{k-i}}\binom{k-1}{j}\binom{k-s+i}{k-i}\binom{k-j}{s-i-j}\times\\\binom{2k-2j-2p}{k-j-p}\left[\binom{2p-s+2j}{p-i+j}-2\binom{2p-s+2j-1}{p-i+j-1}\right] t^{s}[D^{i+1}h_{k}](1+t).
\end{multline}
Owing to the presence of the binomial coefficient $\binom{2k-2j-2p}{k-j-p}$, the upper limit of summation in $p$
 can be extended to infinity. Similarly, the binomial coefficients $\binom{2p-s+2j}{p-i+j}$ and $\binom{2p-s+2j-1}{p-i+j-1}$ allow us to extend the range of $i$. Now, since the limits are independent of $p$,  we can freely push the summation on $p$ inside. 
 \begin{multline}
2^{3k+1}k!t^{2k+1}f_{+}(t)=\sum_{j=0}^{k-1}\sum_{i=j}^{k}\sum_{s=i+j}^{2i}\frac{(-1)^{i+j+1}(k-i)!}{2^{k-i}}\binom{k-1}{j}\binom{k-s+i}{k-i}\binom{k-j}{s-i-j}\times\\\underbrace{\sum_{p=0}^{\infty}\binom{2k-2j-2p}{k-j-p}\left[\binom{2p-s+2j}{p-i+j}-2\binom{2p-s+2j-1}{p-i+j-1}\right]}_{S_2} t^{s}[D^{i+1}h_{k}](1+t).
\end{multline}
Using the result of Lemma \ref{com2:lemma} for $S_2$, we obtain
\begin{multline*}
    2^{3k+1}k!t^{2k+1}f_{+}(t)=\sum_{j=0}^{k}\sum_{i=j}^{k}\sum_{s=i+j}^{2i}\frac{(-1)^{i+j+1}(k-i)!}{2^{k-i}}\binom{k-1}{j}\binom{k-s+i}{k-i}\binom{k-j}{s-i-j}\binom{2k-s}{k-i}\times\\t^{s}[D^{i+1}h_{k}](1+t).
\end{multline*}
Since the binomial coefficient $\binom{k-j}{s-i-j}$ vanishes for $j>s-i$, the upper limit of summation in $j$ may be extended to $k-1$. Interchanging the order of summation, we obtain
\begin{multline*}
2^{3k+1}k!t^{2k+1}f_{+}=\\\sum_{i=0}^{k}\sum_{s=i}^{2i}\frac{(-1)^{i+1}(k-i)!}{2^{k-i}}\binom{k-s+i}{k-i}\binom{2k-s}{k-i}\left[\sum_{j=0}^{k-1}(-1)^{j}\binom{k-1}{j}\binom{k-j}{s-i-j}\right]t^{s}[D^{i+1}h_{k}](1+t).
\end{multline*}
Now consider the inner sum on $j$
\begin{multline*}
    S=\sum_{j=0}^{k-1}(-1)^{j}\binom{k-1}{j}\binom{k-j}{s-i-j}
    =\frac{1}{2\pi i}\int_{\norm{z}=\epsilon}\frac{(1+z)^k}{z^{s-i+1}}\left[\sum_{j=0}^{k-1}(-1)^j\binom{k-1}{j}\left(\frac{z}{1+z}\right)^j\right]\\
    =\frac{1}{2\pi i}\int_{\norm{z}=\epsilon}\frac{(1+z)^k}{z^{s-i+1}}\left[1-\frac{z}{1+z}\right]^{k-1}   =\frac{1}{2\pi i}\int_{\norm{z}=\epsilon}\frac{(1+z)^1}{z^{s-i+1}}=\binom{1}{s-i}.
\end{multline*}
Hence we get
\begin{align*}
2^{3k+1}k!t^{2k+1}f_{+}(t)
&=\sum_{i=0}^{k}\sum_{s=i}^{i+1}\frac{(-1)^{i+1}(k-i)!}{2^{k-i}}\binom{k-s+i}{k-i}\binom{2k-s}{k-i}\binom{1}{s-i}t^{s}[D^{i+1}h_{k}](1+t)=M^{+}(t).
\end{align*}
A similar computation hold for $f_{-}$ as well. 
This completes the proof.
\end{proof}

\section{Proof of Theorem \ref{Mainthm1}}\label{Mthm:step2}
The proof of Theorem \ref{Mainthm1} is carried out in two steps. We first establish that $\dfrac{4^{k}\omega_{2k+3}t^{2k+1}}{\omega_{2k+2}}\mathcal{R}f(t) = h_k(t)$, and then prove the smoothness of $f$.
\subsection{Step 1:} We establish that $\dfrac{4^{k}\omega_{2k+3}t^{2k+1}}{\omega_{2k+2}}\mathcal{R}f(t) = h_k(t)$, and then prove the smoothness of $f$.
 Denote
\begin{align*}
    \widetilde{h}_{k}(t):=\int_{|1-t|}^1uf(u)(Q(t,u))^k\D u,
\end{align*}
where $f$ is given in \eqref{Eq f expression}. Our objective is to prove that $\widetilde{h}_{k}(t)=h_{k}(t)$ for $t\in [0,2]$. For clarity, we first establish the result for $t\in [0,1]$; the case $t\in[1,2]$ follows by an analogous argument. Before proceeding to the proof of the main result, we first justify the following statement:
\subsection*{$\widetilde{h}_{k}(t)$ is well-defined for $t\in[0,2]$ and all $D$-derivatives of $\widetilde{h}_{k}(t)$  vanish at $t=0$}
Let us consider the following integral for $ 0\leq t \leq \epsilon<1$,
\begin{align*}
    \widetilde{h}_{k}(t)=\int\limits_{1-t}^1 uf(u)(Q(t,u))^k\D u,
\end{align*}
where $f$ is given by \eqref{Eq f expression}. That is, 
\[
\wt{h}_k(t)=\frac{(-1)^{k+1}}{2^{3k+1}k!}\int\limits_{1-t}^{1} u \lb \sum_{j=0}^{k-1}(-1)^{j}\binom{k-1}{j}D^{k-j}\left(\frac{[D^{j+1}h_{k}](1+u)-[D^{j+1}h_{k}](1-u)}{u}\right)\rb Q^{k}(t,u)\, \D u. 
\]
Applying integration by parts (IBP) $k-j$ times, we get
\begin{align*}
    \widetilde{h}_{k}(t)=\frac{-1}{2^{3k+1}k!}\sum_{j=0}^{k-1}\binom{k-1}{j}\int_{1-t}^{1}\left([D^{j+1}h_{k}](1+u)-[D^{j+1}h_{k}](1-u)\right)D^{k-j}_u\left\{(Q(t,u))^k\right\} \D u.
\end{align*}
Due to the support condition on $h_k(t)$, we have that  $[D^{j}h_{k}](u)$ is smooth for $u\in [0,1]$ and $\forall j$. So the function $\widetilde{h}_k(t)$ vanishes to infinite order at $0$ and hence is smooth on $[0,\epsilon]$. 
Now we can let $\epsilon \to 1$, and thus obtain the smoothness of $\widetilde{h}_{k}(t)$ for $t \in [0,1]$, where smoothness at the endpoints are considered as one-sided limits. A similar argument applies to the case $t\in[1,2]$. Hence $\widetilde{h}_{k}(t)$ is well defined.
\begin{remark}
    Note that although the one-sided limits of $\widetilde{h}_{k}(t)$ exist at $t=1$, the function may still exhibit a jump discontinuity at this point. 
We will rule out this possibility by proving that $\widetilde{h}_{k}(t)=h_{k}(t)$ for $t\in [0,1]$ and for $[1,2]$. Since $h_{k}$ is smooth on $[0,2]$, it follows that $\widetilde{h}_{k}(t)=h_{k}(t)$ for $t\in [0,2]$ and $\widetilde{h}_{k}(t)$ is smooth on $[0,2]$.
\end{remark}

We  now proceed to establish $\widetilde{h}_{k}(t)=h_{k}(t)$ for $t\in[0,1]$.
Since the $D$-derivatives of $\widetilde{h}_{k}(t)$ and $h_{k}(t)$ vanish at $t=0$, it suffices to establish that
\begin{align}
    \frac{d}{dt}[D^{2k}\widetilde{h}_k](1-t)=\frac{d}{dt}[D^{2k}h_k](1-t)
\end{align}
which consequently implies that $\widetilde{h}_{k}(t)=h_{k}(t)$. Furthermore, in view of \eqref{rangeee}, namely, $[\mathcal{L}_{k}h_{k}](1+t)-[\mathcal{L}_{k}h_{k}](1-t)\equiv 0$, it suffices to establish the following relation:
\begin{align}\label{Tlinear}
    \frac{d}{dt}[D^{2k}\widetilde{h}_k](1-t)=T\left([\mathcal{L}_{k}h_{k}](1+t)-[\mathcal{L}_{k}h_{k}](1-t)\right)+\frac{d}{dt}[D^{2k}h_k](1-t),
\end{align}
for a suitable linear differential operator $T$ as $T\left([\mathcal{L}_{k}h_{k}](1+t)-[\mathcal{L}_{k}h_{k}](1-t)\right)\equiv 0$. In the following proposition, we obtain a suitable linear differential operator $T$ satisfying the above relation.
\begin{proposition}
    For $k\in\mathbb{N}$, the following relation holds:
    \begin{align}\label{Tlinear}
    \frac{d}{dt}[D^{2k}\widetilde{h}_k](1-t)=T\left([\mathcal{L}_{k}h_{k}](1+t)-[\mathcal{L}_{k}h_{k}](1-t)\right)+\frac{d}{dt}[D^{2k}h_k](1-t),
\end{align}
where the linear differential operator $T$ is given by
\begin{align*}
    T:
    &=\frac{(-1)^{k+1}}{(1-t)^{2k}}\sum_{l=0}^{k} \frac{(k+l-1)!}{(k-l)!l!2^{l}}\left( l+ \frac{(k-l)}{2}t\right)(1-t)^{k-l} D^{k-l+1}.
\end{align*}
\end{proposition}
\begin{proof}
Recall from  Theorem \ref{ODE:inversion}, the following relation: 
\begin{align}
       &\frac{d}{dt}[D^{2k}\widetilde{h}_k](1-t)\\
			&=\frac{(-1)^{k}k!2^{3k}}{(1-t)^{2k}}   [\mathcal{L}_{k}(t^{2k+1}f)](t)=\frac{(-1)^{k}}{2(1-t)^{2k}}  [\mathcal{L}_{k}(k!2^{3k+1}t^{2k+1}f)](t)\\
            &=\frac{(-1)^{k}}{2(1-t)^{2k}} \Big{[}\Lc_k\left(k! 2^{3k+1} t^{2k+1} f_{+}-k! 2^{3k+1} t^{2k+1} f_{-}\right)\Big{]}(t)\\
            &=\frac{(-1)^{k}}{2(1-t)^{2k}} \Big{[}\Lc_k\left(k! 2^{3k+1} (\cdot)^{2k+1} f_{+}-k! 2^{3k+1} (\cdot)^{2k+1} f_{-}+D([\Lc_k h_k](1+\cdot)-[\Lc_k h_k](1-\cdot))\right)\Big{]}(t)\\
            &-\frac{(-1)^{k}}{2(1-t)^{2k}}\mathcal{L}_{k}D\left\{[\mathcal{L}_{k}h_{k}](1+t)-[\mathcal{L}_{k}h_{k}](1-t)\right\}.
\end{align}
In the last equality, we added and subtracted $D\left([\Lc_k h_k](1+\cdot)-[\Lc_k h_k](1-\cdot)\right)(t)$. 
Using \eqref{claim1} and \eqref{claim2}, we get
\begin{multline}
    \frac{d}{dt}[D^{2k}\widetilde{h}_k](1-t)=\frac{(-1)^{k}}{2(1-t)^{2k}}   \mathcal{L}_{k}\Big{[}(t^2-1)([\mathcal{L}_{k-1}D^2h_k](1+t)-[\mathcal{L}_{k-1}D^2h_k](1-t))\Big{]}
    -\\\frac{(-1)^{k}}{2(1-t)^{2k}}\mathcal{L}_{k}D\left\{[\mathcal{L}_{k}h_{k}](1+t)-[\mathcal{L}_{k}h_{k}](1-t)\right\}.
\end{multline}
Using Proposition \ref{Lem:Lktointegral}, we get
\begin{multline}\label{Mthm:rel1}
   \frac{d}{dt}[D^{2k}\widetilde{h}_k](1-t)=-\frac{1}{2^{3k-2}(k-1)!(1-t)^{2k}}\mathcal{L}_{k}\left((t^2-1)\int_{1-t}^{1+t}\frac{d}{du}[D^{2k}h_{k}](u)(Q(t,u))^{k-1}\D u\right)\\-\frac{1}{2^{3k+1}k!(1-t)^{2k}}\mathcal{L}_kD\left(\int_{1-t}^{1+t}\frac{d}{du}[D^{2k}h_{k}](u)(Q(t,u))^k\D u\right).
\end{multline}
Using Proposition \ref{Lem:Lktointegral}, the RHS of \eqref{Tlinear} takes the form
\begin{multline}\label{Mthm:rel2}
    T\left([\mathcal{L}_{k}h_{k}](1+t)-[\mathcal{L}_{k}h_{k}](1-t)\right)+\frac{d}{dt}[D^{2k}h_k](1-t)
    =\\T\left(\frac{(-1)^k}{2^{3k}k!}\int_{1-t}^{1+t}\frac{d}{du}[D^{2k}h_{k}](u)(Q(t,u))^k\D u\right)+\frac{d}{dt}[D^{2k}h_{k}](1-t).
\end{multline}
In order to show the equality in  \eqref{Tlinear}, it suffices to prove that RHS of \eqref{Mthm:rel1} and \eqref{Mthm:rel2} are equal i.e.,
\begin{multline}\label{Mthm:rel3}
     -\frac{1}{2^{3k-2}(k-1)!(1-t)^{2k}}\mathcal{L}_{k}\left((t^2-1)\int_{1-t}^{1+t}\frac{d}{du}[D^{2k}h_{k}](u)(Q(t,u))^{k-1}\D u\right)\\-\frac{1}{2^{3k+1}k!(1-t)^{2k}}\mathcal{L}_kD\left(\int_{1-t}^{1+t}\frac{d}{du}[D^{2k}h_{k}](u)(Q(t,u))^k\D u\right)=\\\frac{(-1)^k}{2^{3k}k!}T\left(\int_{1-t}^{1+t}\frac{d}{du}[D^{2k}h_{k}](u)(Q(t,u))^k\D u\right)+\frac{d}{dt}[D^{2k}h_{k}](1-t).
\end{multline}
Let us first examine the boundary terms on both sides. The boundary terms arising from LHS are
\begin{align*}
    &-\frac{(t^2-1)(1-t)^{k}}{2^{3k-2}(k-1)!(1-t)^{2k}t}\times\\
    &\left(\frac{d}{dt}[D^{2k}h_{k}](1+t)2^{3k-3}(k-1)!(1+t)^{k-1}+\frac{d}{dt}[D^{2k}h_{k}](1-t)2^{3k-3}(k-1)!(1-t)^{k-1}\right)\\
    &-\frac{(1-t)^{k}}{2^{3k+1}k!(1-t)^{2k}t}\left(\frac{d}{dt}[D^{2k}h_{k}](1+t)k!2^{3k}(1+t)^k+\frac{d}{dt}[D^{2k}h_{k}](1-t)k!2^{3k}(1-t)^k\right)\\
    &=-\frac{(t^2-1)}{2(1-t)^kt}\left(\frac{d}{dt}[D^{2k}h_{k}](1+t)(1+t)^{k-1}+\frac{d}{dt}[D^{2k}h_{k}](1-t)(1-t)^{k-1}\right)\\
    &-\frac{1}{2(1-t)^kt}\left(\frac{d}{dt}[D^{2k}h_{k}](1+t)(1+t)^k+\frac{d}{dt}[D^{2k}h_{k}](1-t)(1-t)^k\right)\\
  &=-\frac{(1+t)^k}{2(1-t)^k}\frac{d}{dt}[D^{2k}h_{k}](1+t)+\frac{1}{2}\frac{d}{dt}[D^{2k}h_{k}](1-t).
\end{align*}
The boundary terms arising from the  right-hand side are
\begin{align*}
   & \frac{(-1)^k}{2^{3k}k!t}\frac{(-1)^{k+1}}{(1-t)^{2k}2}t(1-t)^k\left(\frac{d}{dt}[D^{2k}h_{k}](1+t)k!2^{3k}(1+t)^{k}+\frac{d}{dt}[D^{2k}h_{k}](1-t)k!2^{3k}(1-t)^k\right)\\
   &=-\frac{(1+t)^k}{2(1-t)^k}\frac{d}{dt}[D^{2k}h_{k}](1+t)-\frac{1}{2}\frac{d}{dt}[D^{2k}h_{k}](1-t).
\end{align*}
Taking into account the boundary terms arising on both sides, establishing \eqref{Mthm:rel3} reduces to proving the following:
\begin{multline*}
   - \frac{1}{2^{3k-2}(k-1)!(1-t)^{2k}}\int_{1-t}^{1+t}\frac{d}{du}[D^{2k}h_{k}](u)\mathcal{L}_k((t^2-1)(Q(t,u))^{k-1})\D u\\-\frac{1}{2^{3k+1}k!(1-t)^{2k}}\int_{1-t}^{1+t}\frac{d}{du}[D^{2k}h_{k}](u)\mathcal{L}_kD(Q(t,u))^k\D u=\frac{(-1)^{k}}{2^{3k}k!}\int_{1-t}^{1+t}\frac{d}{du}[D^{2k}h_{k}](u)T(Q(t,u))^k\D u.
\end{multline*}
That is, it is enough to show that 
\begin{equation}\label{mthm:rel4}
      \int_{1-t}^{1+t}\frac{d}{du}[D^{2k}h_{k}](u)\left[\mathcal{L}_k((t^2-1)Q(t,u))^{k-1}+\left(\frac{1}{2^3k}\mathcal{L}_kD+\frac{(-1)^{k}(1-t)^{2k}}{2^2k}T\right)(Q(t,u))^k\right]\D u=0.
\end{equation}
    By a straightforward check, we have 
    \begin{align*}
        \frac{1}{2^3k}\mathcal{L}_kD+\frac{(-1)^{k}(1-t)^{2k}}{2^2k}T=\frac{(1-t)^{2}}{2^{3}k}\mathcal{L}_{k-1}D^{2}.
    \end{align*}
Thus \eqref{mthm:rel4} takes the form
\begin{equation}\label{rightinverse1}
     \int_{1-t}^{1+t}\frac{d}{du}[D^{2k}h_{k}](u)
\underbrace{\left[\mathcal{L}_k\left\{(t^2-1)(Q(t,u))^{k-1}\right\}+\frac{(1-t)^2}{2}\mathcal{L}_{k-1}D\left\{(Q(t,u))^{k-1}(1+u^2-t^2)\right\}\right]}_{F(t,u)}\D u=0.
\end{equation}
Thus, proving \eqref{Tlinear}  reduces to establishing \eqref{rightinverse1}. The latter follows  once we show that 
$F(t,u)\equiv 0$, which is precisely the statement of Lemma \ref{mainlemma}. This completes the proof.
\end{proof}
\begin{remark}
    In order to prove $\widetilde{h}_{k}(t)=h_{k}(t)$ for $t\in[1,2]$, the relation \eqref{Tlinear} in the above proposition takes the form
    \begin{align}
    \frac{d}{dt}[D^{2k}\widetilde{h}_k](1+t)=T\left([\mathcal{L}_{k}h_{k}](1+t)-[\mathcal{L}_{k}h_{k}](1-t)\right)+\frac{d}{dt}[D^{2k}h_k](1+t),
\end{align}
where the linear operator $T$ is given by
\begin{align*}
    T:
    &=\frac{(-1)^{k+1}}{(1+t)^{2k}}\sum_{l=0}^{k} \frac{(k+l-1)!}{(k-l)!l!2^{l}}\left( -l+ \frac{(k-l)}{2}t\right)(1+t)^{k-l} D^{k-l+1}.
\end{align*}
The proof will follow analogously by using \eqref{ODE:op2} instead of \eqref{ODE:op}.
\end{remark}

\subsection{Step 2: $f$ is smooth}\label{Mthm:step3}

   From Proposition \ref{equiv:repoff} for $k\geq 1$ and with the representation for $f$ given in Theorem \ref{Mainthm1} for the case $n=3$ (corresponding to $k=0$),   the function $f$ has the form
    \begin{align}\label{fexp:smooth}
               f(t)=\frac{1}{2^{3k+1}k!t^{2k+1}} M(t),
           \end{align}
           where $M$ is  a smooth odd function. Hence $f$ is an even function. For the case $n=3$, utilizing the range condition from Theorem \ref{RC:Radial}, we have that $f$ is smooth. 
           
           From now on, we will focus on the case $k\geq 1$. To show that 
$f$ is smooth, it suffices to prove that 
$M(t)$
 vanishes to order at least $2k+1$ at $t=0$, since $M$ is smooth.

    Our first claim is to show that $ \left[\frac{\D}{\D t}D^{2k}h_{k}\right](1)=0$.
        Recalling the range condition from \cite[Theorem 2.1]{AAKN2} we have,
    \begin{align*}
        \frac{1}{t^{n-2}}\int_{1-t}^{1+t}\frac{d}{du}[D^{2k}h_{k}](u)(Q(t,u))^k \D u=0
       \implies \int_{1-t}^{1+t}\frac{d}{du}[D^{2k}h_{k}](u)(Q(t,u))^k \D u=0.
    \end{align*}
    Differentiating the above expression $k$ times w.r.t. $t$, and noting that no boundary terms arise, we apply the Faà di Bruno formula to obtain
    \begin{align*}
        \int_{1-t}^{1+t}\frac{d}{du}[D^{2k}h_{k}](u)\left(\sum_{q \geq k / 2}^k  (-1)^{k-q}k!4^{q}\binom{q}{2q-k}\binom{k}{q} [Q(t,u)]^{k-q}(u^2+1-t^2)^{2q-k}\right)\D u=0.
    \end{align*}
    If we differentiate once more, a boundary term arises in the above expression only in the case $q=k$.
    The boundary term corresponding to $q=k$ after differentiation is 
    \begin{equation*}
        \frac{d}{dt}[D^{2k}h_{k}](1+t)k!4^k(2(1+t))^k+\frac{d}{dt}[D^{2k}h_{k}](1-t)k!4^k(2(1-t))^k
    \end{equation*}
    Now if we take $t \to 0$ all the integrals vanish and using the above expression of boundary term we obtain
    \begin{equation}
        \frac{d}{dt}[D^{2k}h_{k}](1)=0.
    \end{equation}
Coming back to the expression \eqref{fexp:smooth}.
Since $M$ is a smooth odd function, $M^{(2j)}(0)=0$ for all $j\in \mathbb{N}$. Moreover, there exists $r\in \mathbb{N}$ and smooth function $\phi$ such that
\begin{equation}\label{M1:exp}
    M(t)=t^{r}\phi(t), \quad \text{ where } \phi(0)\neq 0.
\end{equation}
Note that since $M$ is an odd and $\phi(0)\neq 0$, $r$ cannot be even. 
\begin{align*}
           \left[\frac{\D}{\D t}D^{2k}h_{k}\right](1-t)
			=\frac{(-1)^kk!2^{3k}}{(1-t)^{2k}}[\mathcal{L}_k(t^{2k+1}f)](t)
            &=\frac{(-1)^k}{2(1-t)^{2k}}[\mathcal{L}_k(M)](t)\\
            &=\frac{(-1)^k}{2(1-t)^{2k}}\left[\sum_{l=0}^{k}\frac{(k+l)!}{(k-l)!l!2^l}(1-t)^{k-l}(D^{k-l}M)(t)\right]
\end{align*}
This implies
\begin{equation}
     \frac{(-1)^{k}(1-t)^k}{2}\left[\frac{\D}{\D t}D^{2k}h_{k}\right](1-t)=\sum_{l=0}^{k}\frac{(k+l)!}{(k-l)!l!2^l}(1-t)^{-l}(D^{k-l}M)(t)=\sum_{l=0}^{k} c_{k,l}(t)[D^{k-l}M](t),
\end{equation}
where 
$$c_{k,l}(t)=\frac{(k+l)!}{(k-l)!l!2^l}\frac{1}{(1-t)^l}.$$
Now using the fact that $\frac{d}{dt}[D^{2k}h_{k}](1)=0$, Leibniz rule and Proposition \ref{fder}, we get
\begin{align*}
    0&=\lim_{t\to 0}\left(\sum_{l=0}^{k} c_{k,l}(t)[D^{k-l}M](t)\right)=\lim_{t\to 0}\left(\sum_{l=0}^{k} c_{k,l}(t)D^{k-l}(t^{r}\phi(t))\right)\\
    &=\lim_{t\to 0}\left(\sum_{l=0}^{k} c_{k,l}(t)\sum_{j=0}^{k-l}\binom{k-l}{j}\frac{r!!}{(r-2k+2l+2j)!!}t^{r-2k+2l+2j}D^{j}\phi(t)\right)\\&=\lim_{t\to 0}\left(\sum_{l=0}^{k} c_{k,l}(t)\sum_{j=0}^{k-l}\binom{k-l}{j}\frac{r!!}{(r-2k+2l+2j)!!}t^{r-2k+2l+2j}\sum_{m=0}^{j}B(m,j)t^{m-2j}\phi^{(m)}(t)\right)
    \\
    &=\lim_{t\to 0}t^{r-2k}\left(\sum_{l=0}^{k} c_{k,l}(t)\sum_{j=0}^{k-l}\binom{k-l}{j}\frac{r!!}{(r-2k+2l+2j)!!}\sum_{m=0}^{j}B(m,j)t^{m+2l}\phi^{(m)}(t)\right).
\end{align*}
Clearly, the only surviving term is the one with $m=l=0$
; hence, we obtain
\begin{equation}
   \lim_{t\to 0}t^{r-2k}\left(\sum_{j=0}^{k}\binom{k}{j}\frac{r!!}{(r-2k+2j)!!}B(0,j)\phi(0)\right)=0.
\end{equation}
Since $B(0,0)=1$ and $B(0,j)=0$ for all $j\geq 1$. Hence
\begin{equation}
   \lim_{t\to 0}t^{r-2k}\left(\frac{r!!}{(r-2k)!!}\phi(0)\right)=0.
\end{equation}
Since $\phi(0)\neq 0$, therefore $r$ must be  odd and $r>2k$. Hence $M(t)=t^{2k+1}\phi(t)$. Hence $f$ is smooth.
This completes the proof of Theorem \ref{Mainthm1}.
\begin{remark}
    Having established the smoothness of $f$, we can use the ODE in \eqref{ODE:inversion} to iteratively determine $f$ and all its derivatives at $0$. For instance, $f(0)$ is explicitly given by
\begin{equation*}
    f(0) = \frac{(-1)^{k+1}}{2^{2k}(2k+1)!}\left[\frac{\mathrm{d}^2}{\mathrm{d}t^2}D^{2k}h_k\right](1).
\end{equation*}
A similar argument can be applied for general functions as well. 
\end{remark}

\section{General case}\label{sec:genfunction}

The spherical harmonics expansion of $g\in C_c^{\infty}(\Sb^{n-1}\times (0,2))$ is given by    
\begin{equation*}
     g(\theta,t)=\sum_{q=0}^{\infty}\sum_{s=1}^{d_{q}} g_{q,s}(t)Y_{q,s}(\theta),
\end{equation*}
where $g_{q,s}\in C_{c}^{\infty}((0,2))$. Let 
    $h_{q,s}(t):=\dfrac{\omega_{2k+3}t^{2k+1}}{\omega_{2k+2}}g_{q, s}(t)$. Let us assume that $h_{q,s}$ satisfies the following two conditions for each $q,s$. There is a function $\phi_{q,s}(t)\in C_c^{\infty}((0,2))$ such that
\begin{equation}\label{ss}
    h_{q,s}(t)= [D^{q}\phi_{q,s}](t)\mbox{ and } \left[\mathcal{L}_{q+k} \phi_{q,s}\right](1-t)=\left[\mathcal{L}_{q+k}\phi_{q,s}\right](1+t).
\end{equation} 
We will construct a candidate ${f}_{q, s}(t)\in C_c^{\infty}([0,1))$ such that $f(x)\in C_c^{\infty}(\Bb)$ defined by $f(x):= \sum\limits_{q=0}^{\infty}\sum\limits_{s=1}^{d_{q}} f_{q,s}(t)Y_{q,s}(\theta),  x=t\theta, $ such that  
\begin{equation}\label{phi(q,s)(t)}
    \phi_{q,s}(t)=\frac{k!}{ 2^{q}4^{q+k}(q+k)!}\int_{\norm{1-t}}^1 u^{1-q}{f}_{q, s}(u)(Q(t,u))^{k+q}\D u.
\end{equation}
We define 
\begin{equation}\label{wtfqs(t)}
\begin{aligned}
    \widetilde{{f}}_{q,s}(t):&=\frac{(-1)^{k+q-1}(t^2-1)}{t^{2k+2q+1}2^{4k+3q-2}k!(k+q-1)!}\int_{1-t}^{1+t}\frac{d}{du}[D^{2k+2q}\phi_{q,s}](u)(Q(t,u))^{k+q-1}\D u\\
    &=\frac{(-1)^{k+q-1}(t^2-1)}{t^{2k+2q+1}2^{4k+3q-2}k!(k+q-1)!}\int_{1-t}^{1+t}\frac{d}{du}[D^{2k+q}h_{q,s}](u)(Q(t,u))^{k+q-1}\D u.
    \end{aligned} 
\end{equation}
The second equality above uses the relation $h_{q,s}(t)=[D^{q}\phi_{q,s}](t)$ from \eqref{ss}.
With this define 
\begin{equation}
    \widetilde{\phi}_{q,s}(t):=\int_{\norm{1-t}}^1 u \left(\frac{k!}{ 2^{q}4^{q+k}(q+k)!}\widetilde{f}_{q, s}(u)\right)(Q(t,u))^{k+q}\D u. 
\end{equation}
Using arguments similar to those given for the radial case, $\widetilde{\phi}_{q,s}(t)$ is a smooth function on $[0,2]$, $\widetilde{\phi}_{q,s}(t)=\phi_{q,s}(t)$ and $\widetilde{f}_{q, s}$ is smooth across the origin.

Finally, let us define 
\[
f_{q,s}(t):=t^{q} \wt{f}_{q,s}(t).
\]
Since $\wt{\phi}_{q,s}(t)=\phi_{q,s}(t)$, we see that \eqref{phi(q,s)(t)} is satisfied. 

\begin{proposition}
    The function $f$ defined by 
    \begin{equation}\label{fseries}
    f(x):= \sum\limits_{q=0}^{\infty}\sum\limits_{s=1}^{d_{q}} f_{q,s}(t)Y_{q,s}(\theta), \quad x=t\theta
\end{equation}
is in $C_c^{\infty}(\Bb)$. 
\end{proposition}

\bpr 

First we prove that the right hand side of \eqref{fseries}
 converges uniformly for $x$ such that  $0<|x|<1-\delta$ for $\delta>0$ small enough. It is enough to show uniform convergence in this region since for $|x|>1-\delta$, due to the support condition on $h$, $\mbox{supp}(f_{q,s}(t))\subset [0,1)$ (this is straightforward to see using \eqref{Eq f expression}). We have 
 \[
 f_{q,s}(t)=\frac{(-1)^{k+q-1}(t^2-1)}{t^{2k+q+1}2^{4k+3q-2}k!(k+q-1)!}\int_{1-t}^{1+t}\frac{d}{du}[D^{2k+q}h_{q,s}](u)(Q(t,u))^{k+q-1}\D u.
 \]
We perform a change of variable $u^2=1+t^2+2ty$.  Then 
\begin{align*}
    Q(t,u)=((1+t)^2-u^2)(u^2-(1-t)^2)=4t^2(1-y^{2}).
\end{align*}
The derivative with respect to $u$ changes as 
\[
    \frac{d}{du}=\frac{d}{dy}\frac{dy}{du}=\frac{d}{dy}\frac{u}{t}.
    \]
    Hence 
    \Beq \label{Dexp:f1}
    D=\frac{1}{u}\frac{d}{du}=\frac{1}{t}\frac{d}{dy}.
\Eeq 
Since $1-t\leq u \leq 1+t$ and $t\in (0,1)$, we have $y\in[-1,1]$. 
Now 
\[
\begin{aligned}
     D^{q-2}((Q(t,u)^{k+q-1})&=\frac{1}{t^{q-2}}\frac{d^{q-2}}{dy^{q-2}}\left\{2^{2k+2q-2}t^{2k+2q-2}(1-y^2)^{q+k-1}\right\}\\
     &=2^{2k+2q-2}t^{2k+q}\frac{d^{q-2}}{dy^{q-2}}\left\{(1-y^2)^{q+k-1}\right\}.
\end{aligned}
\]
We then have 
\[
\begin{aligned}
f_{q,s}(t)&=\frac{(-1)^{k-1}(t^2-1)2^{2k+2q-2}t^{2k+q+1}}{t^{2k+q+1}2^{4k+3q-2}k!(k+q-1)!}\int_{-1}^{1}[D^{2k+3}h_{q,s}](\sqrt{1+t^2+2ty})\frac{d^{q-2}}{dy^{q-2}}\left\{(1-y^2)^{q+k-1}\right\}\D y\\
&=\frac{(-1)^{k-1}(t^2-1)}{2^{2k+q}k!(k+q-1)!}\int_{-1}^{1}[D^{2k+3}h_{q,s}](\sqrt{1+t^2+2ty})\frac{d^{q-2}}{dy^{q-2}}\left\{(1-y^2)^{q+k-1}\right\}\D y.
\end{aligned} 
\]
Recall the following formula for Gegenbauer polynomials:
\begin{align*}
    \frac{d^m}{dy^m}(1-y^2)^{m+\alpha-\frac{1}{2}}=\frac{1}{K}C_{m}^{\alpha}(y)(1-y^2)^{\alpha-\frac{1}{2}}\quad \text{where }K=\frac{(-1)^{m}\Gamma\left(\alpha+\frac{1}{2}\right)\Gamma(m+2\alpha)}{2^{m}m!\Gamma(2\alpha)\Gamma\left(m+\alpha+\frac{1}{2}\right)}.
\end{align*}
For $m=q-2, \alpha=k+\frac{3}{2}>0$, we get
\begin{align*}
    \frac{d^{q-2}}{dy^{q-2}}(1-y^2)^{q+k-1}=\frac{2^{q-2}(q-2)!\Gamma(2k+3)\Gamma\left(q+k\right)}{(-1)^{q-2}\Gamma\left(k+2\right)\Gamma(q+2k+1)}C_{q-2}^{k+\frac{3}{2}}(y)(1-y^2)^{k+1}.
\end{align*}
For Gegenbauer polynomials, it is well-known that $\sup\limits_{y\in [-1,1]}|C_{q-2}^{k+\frac{3}{2}}(y)|\leq \frac{\Gamma(q+2k+1)}{\Gamma(2k+3)(q-2)!}$; see \cite[Theorem 7.33.1]{Gabor}.  Using the above bound, we obtain 
\[
\begin{aligned} 
    |\PD_y^{q-2}(1-y^2)^{k+q-1})|&\leq 2^{2k+2q-2}t^{2k+q}\frac{2^{q-2}(q-2)!\Gamma(2k+3)\Gamma\left(q+k\right)}{\Gamma\left(k+2\right)\Gamma(q+2k+1)}\frac{\Gamma(q+2k+1)}{\Gamma(2k+3)(q-2)!}\\&=2^{q-2}\frac{(k+q-1)!}{(k+1)!}.
\end{aligned}
\]
Combining all these, we have 
\begin{align}\label{fqs-final estimate}
    |f_{q,s}(t)|
    &\leq \frac{1}{2^{2k+2}k!(k+1)!}\int_{-1}^{1}[D^{2k+3}h_{q,s}](\sqrt{1+t^2+2ty})\, \D y 
    \leq \frac{C_{N,\delta}}{q^N}.
\end{align}
The last inequality follows from Property \ref{fact3} on the spherical harmonics coefficients from Section \ref{sec:prepresults}. We will specify the choice of $N$  later.
Denote by $S_{m}(x)$ the $m^{\mathrm{th}}$ partial sum in the expansion of $f$. We have 
\begin{align*}
   S_m(x)= \sum_{q=0}^{m}\left(\sum_{s=1}^{d_{q}} f_{q,s}(t)Y_{q,s}(\theta)\right)= \sum_{q=0}^{m} a(q,t,\theta)=a(0,t,\theta)+a(1,t,\theta)+\sum_{q=2}^{m}a(q,t,\theta).
\end{align*}
Note that $a(0,t,\theta), a(1,t,\theta)$ are smooth functions for $t\in[0,1]$. Now consider $a(q,t,\theta)$ for $q\geq 2$: 
\begin{equation*}
      |a(q,t,\theta)|\leq d_{q}|f_{q,s}(t)|\max\left\{|Y_{q,s}(\theta)|,\theta\in \mathbb{S}^{n-1},1\leq s\leq d_{q}\right\}
    \lesssim d_{q}^\frac{3}{2}|f_{q,s}(t)|\lesssim \left(q^{2k+1}\right)^{\frac{3}{2}}\frac{C_{N,\delta}}{q^{N}},
\end{equation*}
where the second and third inequalities follow from \eqref{fact1} and \eqref{fact2} combined with the estimate for $f_{q,s}(t)$ obtained in \eqref{fqs-final estimate}. 
Now choosing $N=3k+3$, we obtain $|a(q,t,\theta)|\lesssim q^{-\frac{3}{2}}$. Hence the series  \eqref{fseries} converges uniformly for $|x|\in(0,1)$. 
Recalling that we have set $f(0)=f_{0,1}(0) Y_{0,1}(\theta)=\lim\limits_{x\to 0} S_m(x).$ 
We obtain that the function $f$ defined by \eqref{fseries} is a continuous and compactly supported function on $\Bb$. 

Next we prove that for any multi-index $\A$, $\frac{\PD^{\A}}{\PD x^{\A}} S_{m}(x)$ converges uniformly on $0<\norm{x}<1-\delta$. 
Let $x=t\theta$. In some open subset of $\Sb^{n-1}$ if $\theta_1,\cdots,\theta_{n-1}$ serve as coordinates, then we have 
\begin{align}\label{change of variable1}
    \frac{\PD}{\PD x_i}&=\theta_i\frac{\PD}{\PD t}+\frac{1}{t}\sum_{j=1}^{n-1}(\delta_{ij}-\theta_{i}\theta_j)\frac{\PD}{\PD \theta_j},~~\text{for}~1\leq i\leq n-1,\\
   \frac{\PD}{\PD x_n}&=(1-\theta_1^2-\cdots -\theta_{n-1}^2)^{\frac{1}{2}}\left(\frac{\PD}{\PD t}-\frac{1}{t}\sum_{j=1}^{n-1}\theta_j\frac{\PD}{\PD\theta_j}\right).\label{change of variable2}
\end{align}
Note that the sequence of partial sums  $S_{m}(x)$ 
 is a smooth function. Indeed, since   $f_{q,s}(t)=t^q\wt{f}_{q, s}(t)$ and $\wt{f}_{q,s}(t)$ is a smooth function and $|x|^qY_{q,s}\left(\frac{x}{|x|}\right)=P_{q,s}(x)$ where $P_{q,s}(x)$ is a homogeneous polynomial in $\Rb^n$, it follows that
 \begin{equation}
S_m(x)=\sum_{q=0}^{m}\sum_{s=1}^{d_{q}} f_{q,s}(|x|)Y_{q,s}\left(\frac{x}{|x|}\right)=\sum_{q=0}^{m}\sum_{s=1}^{d_{q}} \widetilde{f}_{q,s}(|x|)|x|^qY_{q,s}\left(\frac{x}{|x|}\right)
      =\sum_{q=0}^{m}\sum_{s=1}^{d_{q}} \widetilde{f}_{q,s}(|x|)P_{q,s}(x).
 \end{equation}
Note that $\widetilde{f}_{q, s}(t)$ is an even function of $t$, and therefore $f_{q,s}$ is a function of $t^2=|x|^2$; see \cite{whitney43}. Hence all order derivatives of $S_{m}(x)$ exist at $x=0$, showing that $S_m(x)$ is a smooth function on $\Bb$.   

As before, it suffices to establish in this range due to the support condition on $h$. We have 
\begin{align}
    \PD^{\A}(S_m(x))= \sum_{\beta\leq \A} {\A \choose \beta}\sum_{q=0}^{n}\sum_{s=1}^{d_{q}}  \PD^{\beta}(f_{q,s}(|x|)) \PD^{\A-\beta}\left(Y_{q,s}\left(\frac{x}{|x|}\right)\right)\label{nonsingular2}.
\end{align}
We derive an estimate for $\PD^{\beta}(f_{q,s}(|x|))$ and for the derivative acting on the spherical harmonics we have the estimate \cite[Theorem 4(b)]{Seeley}. 
We consider $f_{q,s}(t)$. Having chosen $\A$, we consider a $q$ large enough (depending on $|\A|$). Based on \eqref{change of variable1} and \eqref{change of variable2}, we see that each $x_i$ derivative introduces a $t^{-1}$ term. Hence at most we obtain a $t^{-|\A|}$ term. This can be handled by suitable integration by parts. Similar to what we did above, we have 
\begin{align}
    f_{q,s}(t)
    =&\frac{(-1)^{k-1}(t^2-1)}{t^{2k+q+1}2^{4k+3q-2}k!(k+q-1)!}\int_{1-t}^{1+t}u[D^{2k+l+1}h_{q,s}](u)D_u^{q-l}(Q(t,u))^{k+q-1}\D u\\
    =&\frac{(-1)^{k-1}(t^2-1)2^{2k+2q-2}t^{2k+2q-2}}{t^{2k+q}2^{4k+3q-1}k!(k+q-1)!}\int_{-1}^{1}[D^{2k+l+1}h_{q,s}](\sqrt{1+t^2+2ty})\frac{1}{t^{q-l}}\frac{d^{q-l}}{dy^{q-l}}\left\{(1-y^2)^{q+k-1}\right\}\D y\\
    =&\frac{(-1)^{k-1}(t^2-1)t^{l-2}}{2^{2k+q+1}k!(k+q-1)!}\int_{-1}^{1}[D^{2k+l+1}h_{q,s}](\sqrt{1+t^2+2ty})\left[\frac{d^{q-l}}{dy^{q-l}}\left\{(1-y^2)^{q+k-1}\right\}\right]\D y.\label{nonsingular1}
\end{align}
The choice of $q$ can be made large enough and the choice of $l$ in \eqref{nonsingular1} above is such that $l<q$ and $l>|\A|+2$. 
Each $x_i$ derivative introduces a derivative with respect to $t$. If we differentiate \eqref{nonsingular1} w.r.t. $t$ repeatedly,  then we obtain certain powers of $1+t^2+2ty$ in the denominator, which can be bounded by
\begin{align*}
   \left|{\frac{1}{(1+t^2+2ty)^j}}\right|\leq \frac{1}{(1-t)^{\frac{j}{2}}}\leq \frac{1}{\delta^{\frac{j}{2}}}.
\end{align*}
Then using the similar estimate as in the continuous case (choose $m=q-l$ and $\alpha=k+(l-\frac{1}{2})>0$) we can bound \eqref{nonsingular2} by a  finite sum of  convergent series. Using Weierstrass $M$-test, we have that the derivatives of $S_m(x)$ converges uniformly for $0<|x|<1-\delta$. As $x\to 0$, $\PD^{\A} S_m(x)$ exists (say) $A_m$. Then $A_m$ converges as $m\to \infty$ and setting this as the value at $x=0$, we have that the limiting function $f(x):=\lim\limits_{m\to \infty} S_m(x)$ is smooth. That is, $f(x)\in C_c^{\infty}(\Bb)$.
\epr 
\begin{remark}
    In the case $k=0$ with $q \ge 1$, where $f_{q,s}$ is defined as in \eqref{General f expression k=0 case}, a convergence analysis similar to the one described above holds. In the case $q=0$, note that \eqref{ss} yield $h_{0,1}(t) = \phi_{0,1}(t)$, which satisfies $h_{0,1}(1-t) = h_{0,1}(1+t)$. If we define $f_{0,1}$ as in \eqref{General f expression k=0, q=0 case}, i.e.,
\[
f_{0,1}(t) = \frac{h_{0,1}^{\prime}(1-t) - h_{0,1}^{\prime}(1+t)}{2t},
\] then it satisfies condition \eqref{phi(q,s)(t)}, since for $0 \le t \le 1$ we have
\begin{align}
    \int_{1-t}^1 u f_{0,1}(u) \, \mathrm{d}u 
    &= \frac{1}{2} \int_{1-t}^1 \left[ -\frac{\mathrm{d}}{\mathrm{d}u}\left(h_{0,1}'(1-u)\right) - \frac{\mathrm{d}}{\mathrm{d}u}\left(h_{0,1}'(1+u)\right) \right] \mathrm{d}u \\
    &= \frac{1}{2} \left[ h_{0,1}'(t) + h_{0,1}'(2-t) \right] = h_{0,1}(t),
\end{align}
where the last equality follows from the range condition. A similar result holds for the case $1 \le t \le 2$.
Moreover, since $f_{0,1}$ is a smooth function, it does not affect the convergence analysis.
\end{remark}

\section{Appendix}\label{Appendix}
\begin{lemma}\label{djDrh:expression}
For $j,r\in\mathbb{N}_{0}$, we have
\begin{align}\
    \frac{d^j}{dt^j}\{[D^rh_{k}](1+t)\}=\sum_{i=0}^{\lfloor{\frac{j}{2}}\rfloor}\frac{j!}{i!(j-2i)!2^i}(1+t)^{j-2i}[D^{r+j-i}h_{k}](1+t).
\end{align}    
\end{lemma}
\begin{proof}
    Trivially true for $j=0$ and $j=1$. Assume the formula holds for $j$: 
    \begin{align*}
    \frac{d^j}{dt^j}\{[D^rh_{k}](1+t)\}&=\sum_{i=0}^{\lfloor{\frac{j}{2}}\rfloor}\frac{j!}{i!(j-2i)!2^i}(1+t)^{j-2i}[D^{r+j-i}h_{k}](1+t).
\end{align*}
We will show that it holds for $j+1$. 
Taking a derivative we get
 \begin{align*}
    &\frac{d^{j+1}}{dt^{j+1}}\{[D^rh_{k}](1+t)\}\\&=\sum_{i=0}^{\lfloor{\frac{j}{2}}\rfloor}\frac{j!(j-2i)}{i!(j-2i)!2^i}(1+t)^{j-2i-1}[D^{r+j-i}h_{k}](1+t)+\sum_{i=0}^{\lfloor{\frac{j}{2}}\rfloor}\frac{j!(j-2i)}{i!(j-2i)!2^i}(1+t)^{j+1-2i}[D^{r+j+1-i}h_{k}](1+t).
\end{align*}
Replacing $i$ by $i-1$ in the first summation
\begin{align*}
    &\frac{d^{j+1}}{dt^{j+1}}[D^rh_{k}](1+t)=\sum_{i=1}^{\lfloor{\frac{j}{2}}\rfloor+1}\frac{j!(j-2i+2)}{(i-1)!(j-2i+2)!2^{i-1}}(1+t)^{j+1-2i}[D^{r+j+1-i}h_{k}](1+t)\\&+\sum_{i=0}^{\lfloor{\frac{j}{2}}\rfloor}\frac{j!}{i!(j-2i)!2^i}(1+t)^{j+1-2i}[D^{r+j+1-i}h_{k}](1+t)\\
    &=(1+t)^{j+1}[D^{r+j+1}h_{k}](1+t)\\
    &+\sum_{i=1}^{\lfloor{\frac{j}{2}}\rfloor}\frac{j!}{(i-1)!(j-2i)!2^{i-1}}\left[\frac{1}{(j-2i+1)}+\frac{1}{2i}\right](1+t)^{j+1-2i}[D^{r+j+1-i}h_{k}](1+t)\\
    &+\left.\frac{j!(j-2i+2)}{(i-1)!(j-2i+2)!2^{i-1}}(1+t)^{j+1-2i}[D^{r+j+1-i}h_{k}](1+t)\right|_{i=\lfloor{\frac{j}{2}}\rfloor+1}\\
    &=(1+t)^{j+1}[D^{r+j+1}h_{k}](1+t)+\sum_{i=1}^{\lfloor{\frac{j}{2}}\rfloor}\frac{(j+1)!}{i!(j-2i+1)!2^{i}}(1+t)^{j+1-2i}[D^{r+j+1-i}h_{k}](1+t)\\
    &+\left.\frac{j!(j-2i+2)}{(i-1)!(j-2i+2)!2^{i-1}}(1+t)^{j+1-2i}[D^{r+j+1-i}h_{k}](1+t)\right|_{i=\lfloor{\frac{j}{2}}\rfloor+1}\\
    &=\sum_{i=0}^{\lfloor{\frac{j+1}{2}}\rfloor}\frac{(j+1)!}{i!(j+1-2i)!2^{i}}(1+t)^{j+1-2i}[D^{r+j+1-i}h_{k}](1+t),
\end{align*}
where the last step follows from the observation that
\begin{equation}
   \left.\frac{j!(j-2i+2)}{(i-1)!(j-2i+2)!2^{i-1}}\right|_{i=\lfloor{\frac{j}{2}}\rfloor+1}=\left.\frac{(j+1)!}{i!(j+1-2i)!2^{i}}\right|_{i=\lfloor{\frac{j+1}{2}}\rfloor}.
\end{equation}
This completes the proof.
\end{proof}
\begin{lemma}\label{Comp1:lemma}
    \begin{equation}
       S:= \sum^{m}_{j=i}(-1)^{j}\binom{j}{i}\binom{2m-1-j}{m-1}\binom{s-1}{j-1}= \frac{m(-1)^i}{i}\binom{s-1}{i-1}\left[\binom{2m-s}{m-i}-2\binom{2m-s-1}{m-i-1}\right].
\end{equation}
\end{lemma}
\begin{proof} Using the relation $\binom{j}{i}\binom{s-1}{j-1}=\frac{j}{i}\binom{j-1}{i-1}\binom{s-1}{j-1}=\frac{j}{i}\binom{s-1}{i-1}\binom{s-i}{j-i}$ and replacing $j$ by $m-j$, the expression of $S$ takes the form
\begin{align}\label{S:expression}
   S:=\frac{(-1)^m}{i}\binom{s-1}{i-1} \sum^{m-i}_{j=0}(-1)^{j}(m-j)\binom{m-1+j}{j}\binom{s-i}{m-j-i}.
\end{align}
Noting that the binomial $\binom{s-i}{m-j-i}$ vanishes for $j>m-i$, we may extend the upper limit of summation in $j$ to infinity.
Using the identity
\begin{align*}
     \sum^{\infty}_{j=0}(m-j)\binom{m-1+j}{j}x^{j}=\frac{m}{(1-x)^m}-\frac{mx}{(1-x)^{m+1}}=m\frac{(1-2x)}{(1-x)^{m+1}}\quad\text{provided }|x|<1.
\end{align*}
Now consider
\begin{align*}
    &\sum^{\infty}_{j=0}(-1)^{j}(m-j)\binom{m-1+j}{j}\binom{s-i}{m-j-i}
    =\frac{1}{2\pi i}\int_{|z|=\epsilon}\frac{(1+z)^{s-i}}{z^{m-i+1}} \sum^{\infty}_{j=0}(m-j)\binom{m-1+j}{j}\left(-{z}\right)^{j}dz\\&=\frac{m}{2\pi i}\int_{|z|=\epsilon}\frac{(1+z)^{s-i}(1+2z)}{z^{m-i+1}(1+z)^{m+1}}dz
    =\frac{m}{2\pi i}\int_{|z|=\epsilon}\frac{(1+2z)}{z^{m-i+1}(1+z)^{m+i-s+1}}dz\\
    &=\frac{m}{2\pi i}\int_{|z|=\epsilon}\frac{1}{z^{m-i+1}}\sum_{p=0}^{\infty}\binom{m+i-s+p}{p}(-z)^{p}dz+\frac{2m}{2\pi i}\int_{|z|=\epsilon}\frac{1}{z^{m-i}}\sum_{p=0}^{\infty}\binom{m+i-s+p}{p}(-z)^{p}dz\\
    &=m(-1)^{m-i}\binom{2m-s}{m-i}+2m(-1)^{m-i-1}\binom{2m-s-1}{m-i-1}.
\end{align*}
Substituting this expression into \eqref{S:expression} yields the desired result, thereby completing the proof.
\end{proof}
\begin{lemma} For $m,r,k\in\mathbb{N}_{0}$, we have
  \begin{align}\label{Dmr:exp}
    D^{m}\{[D^rh_{k}](1\pm t)\}(t)&=\sum_{i=0}^{m}\sum_{s=i}^{2i}C_{\pm}(i,s;m)t^{s-2m}[D^{r+i}h_{k}](1\pm t),
\end{align}
where the constants $C_{\pm}(i,s;m)$ are given by
\begin{align}
    C_{+}(i,s;m)&=\frac{m!(-1)^{m-i}}{i!2^{m-i}}\binom{i}{s-i}\left[\binom{2m-s}{m-i}-2\binom{2m-s-1}{m-i-1}\right]\quad\text{and}\\
      C_{-}(i,s;m)&=\frac{m!(-1)^{s+m-i}}{i!2^{m-i}}\binom{i}{s-i}\left[\binom{2m-s}{m-i}-2\binom{2m-s-1}{m-i-1}\right].
\end{align}
\end{lemma}
 For brevity, we present the derivation only for $D^{m}\{[D^{r}h_{k}](1+t)\}(t)$, since the formula for $D^{m}\{[D^{r}h_{k}](1-t)\}(t)$ follows analogously.
\begin{proof} Using Proposition \ref{fder} for $g(t)=[D^{r}h_{k}](1+t)$, we get
\begin{align}\label{DmDrh:express}
D^{m}\{[D^rh_{k}](1+t)\}(t)&=\sum_{j=0}^{m}(-1)^{m-j}\frac{(m-j)!}{2^{m-j}}\binom{m-1}{j-1}\binom{2m-1-j}{m-1}\frac{1}{t^{2m-j}}\frac{d^j}{dt^j}\{[D^rh_{k}](1+t)\}.
\end{align}
Employing Lemma \ref{djDrh:expression}, we get
\begin{multline*}
    D^{m}\{[D^rh_{k}](1+t)\}(t)
=\sum_{j=1}^{m}\sum_{i=0}^{\lfloor{\frac{j}{2}}\rfloor}(-1)^{m-j}\frac{(m-j)!(2i)!}{i!2^{m-j+i}}\binom{m-1}{j-1}\binom{2m-1-j}{m-1}\binom{j}{j-2i}\times\\\frac{(1+t)^{j-2i}}{t^{2m-j}}[D^{r+j-i}h_{k}](1+t).
\end{multline*}
Expanding $(1+t)^{j-2i}$, we get
\begin{multline*}
     D^{m}\{[D^rh_{k}](1+t)\}(t)
=\sum_{j=1}^{m}\sum_{i=0}^{\lfloor{\frac{j}{2}}\rfloor}\sum_{s=0}^{j-2i}(-1)^{m-j}\frac{(m-j)!(2i)!}{i!2^{m-j+i}}\binom{m-1}{j-1}\binom{2m-1-j}{m-1}\binom{j}{j-2i}\times\\\binom{j-2i}{s}t^{s+j-2m}[D^{r+j-i}h_{k}](1+t).
\end{multline*}
Replacing $i$ by $j-i$ and $s$ by $s-j$, we get
\begin{align*}
   &D^{m}\{[D^rh_{k}](1+t)\}(t)\\
=&\sum_{j=1}^{m}\sum^{j}_{i=j-\lfloor{\frac{j}{2}}\rfloor}\sum_{s=j}^{2i}(-1)^{m-j}\frac{(m-j)!(2j-2i)!}{(j-i)!2^{m-i}}\binom{m-1}{j-1}\binom{2m-1-j}{m-1}\binom{j}{2i-j}\binom{2i-j}{s-j}\times\\
&t^{s-2m}[D^{r+i}h_{k}](1+t)\\
=&(m-1)!(-1)^{m}\frac{1}{2^m}\sum_{j=1}^{m}\sum^{j}_{i=j-\lfloor{\frac{j}{2}}\rfloor}\sum_{s=j}^{2i}(-1)^{j}2^{i}\frac{i!}{(2i-1)!}\binom{2i-1}{2i-j}\binom{j}{i}\binom{2m-1-j}{m-1}\binom{2i-j}{2i-s}\times\\&t^{s-2m}[D^{r+i}h_{k}](1+t)\\
=&(m-1)!(-1)^{m}\frac{1}{2^m}\sum_{j=1}^{m}\sum^{j}_{i=\lfloor{\frac{j+1}{2}}\rfloor}\sum_{s=j}^{2i}(-1)^{j}2^{i}\frac{i!}{(2i-1)!}\binom{2i-1}{2i-s}\binom{j}{i}\binom{2m-1-j}{m-1}\binom{s-1}{j-1}\times\\&t^{s-2m}[D^{r+i}h_{k}](1+t),
\end{align*}
where the last step follows from an application of $\binom{a}{b}\binom{b}{c}=\binom{a}{c}\binom{a-c}{a-b}$.
Using the fact that for $i\in\mathbb{Z}$ and $x\in\mathbb{R}$,  $n<[x]\implies n+1\leq [x]$, we get 
\begin{equation}
    i<\left[\frac{j+1}{2}\right]\implies i+1\leq \left[\frac{j+1}{2}\right]\implies 2i+1\leq j\implies 2i-s\leq 2i-j\leq -1,
\end{equation}
since $j\leq s\leq 2i$. Note that for $i<\left[\frac{j+1}{2}\right]$, the upper index for $s$ is smaller than the lower index of $s$. Hence, the summation in $s$ is empty in this case. Therefore, the term $\binom{2i-1}{2i-s}$ vanishes for $i<\left[\frac{j+1}{2}\right]$. Therefore the lower limit of $s$ can be extended to be $1$.  
Furthermore, the lower limit of $s$ can be extended to $s=i$ due to the term $\binom{s-1}{j-1}$. Upon interchanging the order of summation, we obtain
\begin{multline}\label{Dmr:expre}
    D^{m}\{[D^rh_{k}](1+t)\}(t)=(m-1)!(-1)^{m}\frac{1}{2^m}\sum_{i=1}^{m}\sum_{s=i}^{2i}\frac{i!2^i}{(2i-1)!}\binom{2i-1}{2i-s}\times\\\underbrace{\left(\sum^{m}_{j=i}(-1)^{j}\binom{j}{i}\binom{2m-1-j}{m-1}\binom{s-1}{j-1}\right)}_{S}t^{s-2m}[D^{r+i}h_{k}](1+t).
\end{multline}
Invoking Lemma \ref{Comp1:lemma} for $S$, we obtain
\begin{align*}
     &D^{m}\{[D^rh_{k}](1+t)\}(t)
    =\sum_{i=1}^{m}\sum_{s=i}^{2i}C_{+}(i,s;m)t^{s-2m}[D^{r+i}h_{k}](1+t),
\end{align*}
where
\begin{align}
    C_{+}(i,s;m)=\frac{m!(-1)^{m-i}}{i!2^{m-i}}\binom{i}{s-i}\left[\binom{2m-s}{m-i}-2\binom{2m-s-1}{m-i-1}\right].
\end{align}
Noting that $C(0,0;0)=1$ and $C(0,0;m)=0$ for $m\geq 1$,
 the above expression can be extended to include the case of $m=0$.  A similar computation yields the following formula for $D^{m}\{[D^rh_{k}](1-t)\}(t)$ given above. This completes the proof.
 \end{proof}

\begin{lemma}\label{com2:lemma} For $k,j,i,s\in\mathbb{N}_{0}$, we have
    \begin{equation}
         S:=\sum_{p=0}^{\infty}\binom{2k-2j-2p}{k-j-p}\left[\binom{2p-s+2j}{p-i+j}-2\binom{2p-s+2j-1}{p-i+j-1}\right]=\binom{2k-s}{k-i}.
    \end{equation}
\end{lemma}
\begin{proof}
   Decomposing $S$ as
\begin{align*}
    S=\sum_{p=0}^{\infty}\binom{2k-2j-2p}{k-j-p}\left[\binom{2p-s+2j}{p-i+j}-2\binom{2p-s+2j-1}{p-i+j-1}\right]=S_1-2S_2.
\end{align*}
Then,
\begin{align*}
    S_1=&\sum_{p=0}^{\infty}\binom{2k-2j-2p}{k-j-p}\binom{2p-s+2j}{p-i+j}\\
=&\frac{1}{(2\pi i)^2}\int_{\norm{z}=\epsilon} \frac{(1+z)^{2k-2j}}{z^{k-j+1}}\int_{\norm{w}=\gamma}\frac{(1+w)^{2j-s}}{w^{j-i+1}}\sum_{p=0}^{\infty}\left[\frac{(1+w)^2z}{(1+z)^2w}\right]^{p}dwdz.
\end{align*}
If we choose $\epsilon$ and $\gamma$ such that $\left|\frac{(1+w)^2z}{(1+z)^2w}\right|<1$, then the series $\sum_{p=0}^{\infty}\left[\frac{(1+w)^2z}{(1+z)^2w}\right]^{p}$ converges and  we get
\begin{align*}
      S_{1}=&\frac{1}{(2\pi i)^2}\int_{\norm{z}=\epsilon} \frac{(1+z)^{2k-2j}}{z^{k-j+1}}\int_{\norm{w}=\gamma}\frac{(1+w)^{2j-s}}{w^{j-i+1}}\frac{(1+z)^2w}{(1+z)^2w-(1+w)^2z}dwdz\\
=&\frac{1}{(2\pi i)^2}\int_{\norm{z}=\epsilon} \frac{(1+z)^{2k-2j+2}}{z^{k-j+1}}\int_{\norm{w}=\gamma}\frac{(1+w)^{2j-s}w^{i-j}}{(w-z)(1-zw)}dwdz.
\end{align*}
It is easy to check that
\begin{equation*}
\lim_{w\to z}\frac{(1+w)^{2j-s}w^{i-j}}{(w-z)(1-zw)}=\infty\quad \text{and}\quad 	\lim_{w\to z}\frac{(1+w)^{2j-s}w^{i-j}}{(1-zw)}=\frac{(1+z)^{2j-s}z^{i-j}}{(1+z)(1-z)}\neq 0,\infty,
\end{equation*}
for sufficiently small $\epsilon$ and hence, $w=z$ is a simple pole inside $|w|=\gamma$. Now using the Residue theorem, we get
\begin{align*}
S_{1}=\frac{1}{2\pi i}\int_{\norm{z}=\epsilon} \frac{(1+z)^{2k-2j+2}}{z^{k-j+1}}\frac{(1+z)^{2j-s-1}z^{i-j}}{1-z}dz=\frac{1}{2\pi i}\int_{\norm{z}=\epsilon} \frac{(1+z)^{2k-s+1}}{z^{k-i+1}(1-z)}dz.
\end{align*}
In a similar manner, 	 $S_{2}$ assumes the form
\begin{align*}
    S_2=\sum_{p=0}^{\infty}\binom{2k-2j-2p}{k-j-p}\binom{2p-s+2j-1}{p-i+j-1}
=&\frac{1}{2\pi i}\int_{\norm{z}=\epsilon} \frac{(1+z)^{2k-s}}{z^{k-i}(1-z)}.
\end{align*}
Thus, we have
\begin{align*}
   S= S_1-2S_2=&\frac{1}{2\pi i}\int_{\norm{z}=\epsilon} \frac{(1+z)^{2k-s}}{z^{k-i}(1-z)}\left[\frac{1+z}{z}-2\right]dz
    =\frac{1}{2\pi i}\int_{\norm{z}=\epsilon} \frac{(1+z)^{2k-s}}{z^{k-i+1}}dz=\binom{2k-s}{k-i}.
\end{align*}
\end{proof}
\begin{lemma}\label{mainlemma} For $k\in\mathbb{N}$, we have
\begin{align*} 
    F(t,u)=&\mathcal{L}_k\left\{(t^2-1)(Q(t,u))^{k-1}\right\}+\frac{(1-t)^2}{2}\mathcal{L}_{k-1}D\left\{(Q(t,u))^{k-1}(1+u^2-t^2)\right\}\equiv 0.
\end{align*}
\end{lemma}
For ease of notation, we will write $Q$ instead of $Q(t,u)$ in the following proof.

Before we give the proof of this lemma, let us consider a few special cases. 

With the notation $T_1, T_2, T_3, A, B, C$ from below (see \eqref{Tis} and \eqref{ABC}), we have the following special cases verifying the conclusion of this lemma: 

\begin{enumerate}
    \item $k=1$
    \begin{align*}
        T_{1}&=-BC\\
        T_{2}&=2B\\
        T_{3}&=-B^2
    \end{align*}
Hence
\begin{align*}
    F(t,u)&=T_{1}+T_{2}+T_{3}=-B(B+C-2)=0,
\end{align*}
since $B+C-2=0$.
\item $k=2$
    \begin{align*}
        T_{1}
        &=-3BC(4A-A^2+4AB-4BC)+8B^3C\\
        T_{2}
        &=6B(4A-A^2+4AB-4BC)-8AB^2\\
        T_{3}
        &=-3B^2(4A-A^2+4AB-4BC)+8AB^2-8B^3C.
    \end{align*}
Hence \begin{align*}
    F(t,u)&=T_{1}+T_{2}+T_{3}=-3(B+C-2)(4A-A^2+4AB-4BC)=0, 
    \end{align*}
    due to the factor $B+C-2$ which is $0$. 
\item $k=3$
\begin{align*}        T_{1}
&=-3BC \Big{(}5A^4-40A^3B-40A^3+80A^2B^2+40A^2BC+160A^2B+80A^2\\&-160AB^2C
-64AB^2-160ABC+64B^3C+80B^2C^2\Big{)}\\
&-48A^2B^3C+192AB^4C+192AB^3C-192B^4C^2\\
        T_{2}
&=6B \Big{(}5A^4-40A^3B-40A^3+80A^2B^2+40A^2BC+160A^2B+80A^2-160AB^2C\\
&-64AB^2-160ABC+64B^3C+80B^2C^2\Big{)}\\&+48A^3B^2-192A^2B^3-192A^2B^2+192AB^3C\\
    T_{3}
&=-3B^{2} \Big{(}5A^4-40A^3B-40A^3+80A^2B^2+40A^2BC+160A^2B+80A^2-160AB^2C\\
    &-64AB^2-160ABC+64B^3C+80B^2C^2\Big{)}-48A^3B^2+192A^2B^3+192A^2B^2\\
&-192AB^3C+48A^2B^3C-192AB^4C-192AB^3C+192B^4C^2.
    \end{align*}
    Hence \begin{align*}
    F(t,u)&=T_{1}+T_{2}+T_{3}\\
    &=-3B(B+C-2) \left(5A^4-40A^3B-40A^3+80A^2B^2+40A^2BC\right.\\
    &\left.+160A^2B+80A^2-160AB^2C-64AB^2-160ABC+64B^3C+80^2C^2\right)=0.
\end{align*}
\end{enumerate} 
We now give the proof for general $k$. 
\begin{proof}
Consider
\begin{align*}
    F(t,u)=&\mathcal{L}_k\left\{(t^2-1)Q^{k-1}\right\}+\frac{(1-t)^2}{2}\mathcal{L}_{k-1}D\left\{Q^{k-1}(1+u^2-t^2)\right\}\\
    =&\sum_{l=0}^{k}\frac{(k+l)!}{(k-l)!l!2^l}(1-t)^{k-l}D^{k-l}\left\{(t^2-1)Q^{k-1}\right\}\\
    &+\sum_{l=0}^{k-1}\frac{(k-1+l)!}{(k-1-l)!l!2^{l+1}}(1-t)^{k-l+1} D^{k-l}\left\{Q^{k-1}(u^2+1-t^2)\right\}\\
    &=T_{1}+T_{2}+T_{3},
\end{align*}
where $T_{j},j=1,2,3$ are given by
\Beq\label{Tis}
\begin{aligned}
    T_{1}&:=\sum_{l=0}^{k}\frac{(k+l)!}{(k-l)!l!2^l}(1-t)^{k-l}(t^2-1)D^{k-l}\left\{Q^{k-1}\right\};\\
    T_{2}&:=\sum_{l=0}^{k}\frac{(k+l)!(k-l)}{(k-l)!l!2^{l-1}}(1-t)^{k-l}D^{k-l-1}\left\{Q^{k-1}\right\};\text{ and}\\
    T_{3}&:=\sum_{l=0}^{k-1}\frac{(k-1+l)!}{(k-1-l)!l!2^{l+1}}(1-t)^{k-l+1} D^{k-l}\left\{Q^{k-1}(u^2+1-t^2)\right\}.
\end{aligned}
\Eeq
Introducing the notations 
\Beq\label{ABC}
A:=1+u^2-t^2,~B:=1-t,~\text{and}~C:=1+t,
\Eeq  we have \begin{align*}
    Q(t,u)&=4(1+u^2-t^2)-(1+u^2-t^2)^2-4(1-t)(1+t)=4A-A^2-4BC,
\end{align*}
 $D\left\{Q(t,u)\right\}=4A,\text{ and } D^{2}\left\{Q(t,u)\right\}=-8.$ 
Using the Faa di Bruno's formula, we obtain
\begin{align}
    T_{1}&=
    C\sum_{l=0}^{k}\sum_{q \geq \frac{k-l}{2}}^{k-l}\frac{(k+l)!}{(k-l)!l!2^l}  \frac{(-1)^{k+l+q+1}(k-l)!2^{2q}}{(2 q-k+l)!(k-l-q)!}\frac{(k-1)!}{(k-q)!}(k-q)Q^{k-1-q}A^{2 q-k+l}B^{k-l+1};\label{T1:exp}\\
    T_2&=\sum_{l=0}^{k-1}\sum_{q \geq \frac{k-l-1}{2}}^{k-l-1} \frac{(k+l)!}{(k-l)!l!2^l} \frac{(-1)^{k+l+q+1}(k-l)!2^{2q+1}}{(2 q-k+l+1)!(k-l-1-q)!}\frac{(k-1)!}{(k-1-q)!}Q^{k-1-q}A^{2 q-k+l+1}B^{k-l};\\
    T_{3}&=\sum_{l=0}^{k-1}\sum_{q\geq \frac{k-l-1}{2}}^{k-l}\frac{(k-1+l)!}{(k-1-l)!l!2^{l+1}}\frac{(-1)^{k+l+q}(k-l+1)!2^{2q}}{(2q-k+l+1)!(k-l-q)!}\frac{(k-1)!}{(k-1-q)!}Q^{k-1-q}A^{2 q-k+l+1}B^{k-l+1}.
\end{align}
Our first goal is to express $T_{j},j=1,2,3$ as follows
\begin{align*}
    T_{1}=CM+E_{1},\quad T_{2}=-2M+E_{2},\quad \text{and } T_{3}=BM+E_{3},
\end{align*}
whence we get
\begin{align*}
    F(t,u)=T_{1}+T_{2}+T_{3}=(C-2+B)M+E_{1}+E_{2}+E_{3}=E_{1}+E_{2}+E_{3},
\end{align*}
since $C-2+B\equiv 0$, where the expressions of $M,E_{1},E_{2},$ and $E_{3}$ will be specified later. Then our goal boils down to showing $$F(t,u)=E_{1}+E_{2}+E_{3}\equiv 0.$$
\subsection*{Expressing $T_{1}$ as $ T_{1}:=CM+E_{1}$}
From the expression of $T_1$ we have,
\begin{align*}
    T_{1}:&=
    C\sum_{l=0}^{k}\sum_{q \geq \frac{k-l}{2}}^{k-l}\frac{(k+l)!}{l!2^l}  \frac{(-1)^{k+l+q+1}2^{2q}}{(2 q-k+l)!(k-l-q)!}\frac{(k-1)!}{(k-q)!}(k-q)Q^{k-1-q}A^{2 q-k+l}B^{k-l+1}.
    \end{align*}
Separating the term $l=0$, and separating the terms in first term as: $1=\frac{k-l-q}{k-q}+\frac{l}{k-q},q\neq k$,  we get
    \begin{align}
    T_1=&C\sum_{l=1}^{k}\sum_{q \geq \frac{k-l}{2}}^{k-l}\frac{(k+l)!}{l!2^l}  \frac{(-1)^{k+l+q+1}2^{2q}}{(2 q-k+l)!(k-l-q)!}\frac{(k-1)!}{(k-q)!}(k-q)Q^{k-1-q}A^{2 q-k+l}B^{k-l+1}\\
    &+C\sum_{q \geq \frac{k}{2}}^{k}  \frac{(-1)^{k+q+1}k!2^{2q}}{(2 q-k)!(k-q)!}\frac{(k-1)!}{(k-q)!}(k-q)Q^{k-1-q}A^{2 q-k}B^{k+1}\\
    &\text{splitting first term as $k-q=l+k-l-q$   }\\
    =&
    C\sum_{l=1}^{k}\sum_{q \geq \frac{k-l}{2}}^{k-l}\frac{(k+l)!}{l!2^l}  \frac{(-1)^{k+l+q+1}2^{2q}l}{(2 q-k+l)!(k-l-q)!}\frac{(k-1)!}{(k-q)!}Q^{k-1-q}A^{2 q-k+l}B^{k-l+1}\\
    &+
    C\sum_{l=1}^{k}\sum_{q \geq \frac{k-l}{2}}^{k-l}\frac{(k+l)!}{l!2^l} \frac{(-1)^{k+l+q+1}2^{2q}}{(2 q-k+l)!(k-l-q)!}\frac{(k-1)!}{(k-q)!}(k-l-q)Q^{k-1-q}A^{2 q-k+l}B^{k-l+1}\\
    &+C\sum_{q \geq \frac{k}{2}}^{k}  \frac{(-1)^{k+q+1}k!2^{2q}}{(2 q-k)!(k-q)!}\frac{(k-1)!}{(k-q)!}(k-q)Q^{k-1-q}A^{2 q-k}B^{k+1}\\
    =&
    C\underbrace{\sum_{l=1}^{k}\sum_{q \geq \frac{k-l}{2}}^{k-l}\frac{(k+l)!}{l!2^l}  \frac{(-1)^{k+l+q+1}2^{2q}l}{(2 q-k+l)!(k-l-q)!}\frac{(k-1)!}{(k-q)!}Q^{k-1-q}A^{2 q-k+l}B^{k-l+1}}_{M}\\
    &+
    \underbrace{C\sum_{l=0}^{k}\sum_{q \geq \frac{k-l}{2}}^{k-l}\frac{(k+l)!}{l!2^l} \frac{(-1)^{k+l+q+1}2^{2q}}{(2 q-k+l)!(k-l-q)!}\frac{(k-1)!}{(k-q)!}(k-l-q)Q^{k-1-q}A^{2 q-k+l}B^{k-l+1}}_{E_{1}}\\
      =&:CM+E_{1}.\label{M:exppp}
\end{align}
\subsection*{Expressing $T_{2}$ as $ T_{2}:=-2M+E_{2}$}
\begin{align*}
    T_2:=\sum_{l=0}^{k-1}\sum_{q \geq \frac{k-l-1}{2}}^{k-l-1}\frac{(k+l)!}{l!2^l}  \frac{(-1)^{k+l+q+1}2^{2q+1}}{(2 q-k+l+1)!(k-l-1-q)!}\frac{(k-1)!}{(k-1-q)!}Q^{k-1-q}A^{2 q-k+l+1}B^{k-l}.
\end{align*}
We can write $T_{2}=-2M+E_{2}$ where $M$ is given by \eqref{M:exppp} and $E_{2}:=T_{2}+2M$ which can be simplified to obtain
      \begin{align*}\label{E2:exp}
        E_2:=&-\sum_{l=1}^{k-1}\sum_{q \geq \frac{k-l}{2}}^{k-l}\frac{(k+l-1)!(-1)^{k+l+q}2^{2q+1}(k-1)!}{(l-1)!(2 q-k+l)!(k-l-q)!(k-q)!2^{l}}(2q-k+l)Q^{k-1-q}A^{2 q-k+l}B^{k-l+1}.
    \end{align*}
    \subsection*{Expressing $T_{3}$ as $ T_{3}:=BM+E_{3}$} 
    \begin{equation}
         T_{3}=\sum_{l=0}^{k-1}\sum_{q\geq \frac{k-l-1}{2}}^{k-l}\frac{(k-1+l)!}{(k-1-l)!l!2^{l+1}}\frac{(-1)^{k+l+q}(k-l+1)!2^{2q}}{(2q-k+l+1)!(k-l-q)!}\frac{(k-1)!}{(k-1-q)!}Q^{k-1-q}A^{2 q-k+l+1}B^{k-l+1}.
    \end{equation}
    We can write $T_{3}=BM+E_{3}$ where $M$ is given by \eqref{M:exppp} and $E_{3}:=T_{3}-BM$ which can be simplified to obtain
   \begin{multline*}
        E_{3}= \sum_{l=2}^{k}\sum_{q\geq \frac{k-l}{2}}^{k-l+1}\frac{(k+l-2)!(-1)^{k+l+q}2^{2q-l}(k-1)!\left[(3k+l-2q-2kl-4kq+3k^2-l^2)\right]}{(l-2)!(2q-k+l)!(k-l-q+1)!(k-q)!}\\ \times Q^{k-1-q}A^{2 q-k+l}B^{k-l+2}.
   \end{multline*}
\subsection*{Showing $E_{1}+E_{2}+E_{3}\equiv 0$}
Since $Q=4A-A^{2}-4BC$, substituting $BC=\frac{1}{4}\left(4A-A^{2}-Q\right)$ in $E_{1}$ given by \eqref{M:exppp}, we obtain
\begin{align*}
     E_{1}&= \sum_{l=0}^{k}\sum_{q \geq \frac{k-l}{2}}^{k-l}\frac{(k+l)!}{l!2^l} \frac{(-1)^{k+l+q+1}2^{2q}}{(2 q-k+l)!(k-l-q)!}\frac{(k-1)!}{(k-q)!}(k-l-q)Q^{k-1-q}A^{2 q-k+l+1}B^{k-l}\\
     &+\sum_{l=0}^{k}\sum_{q \geq \frac{k-l}{2}}^{k-l}\frac{(k+l)!}{l!2^l} \frac{(-1)^{k+l+q}2^{2q-2}}{(2 q-k+l)!(k-l-q)!}\frac{(k-1)!}{(k-q)!}(k-l-q)Q^{k-1-q}A^{2 q-k+l+2}B^{k-l}\\
     &+\sum_{l=0}^{k}\sum_{q \geq \frac{k-l}{2}}^{k-l}\frac{(k+l)!}{l!2^l} \frac{(-1)^{k+l+q}2^{2q-2}}{(2 q-k+l)!(k-l-q)!}\frac{(k-1)!}{(k-q)!}(k-l-q)Q^{k-q}A^{2 q-k+l}B^{k-l}.
\end{align*}
Noting that the term corresponds to $l=k$ vanishes, therefore reducing the limits of $k$ to $k-1$.
Replace $l$ by $l-1$ in first term, $l$ by $l-2$ in 2nd term, $q$ by $q+1$ and $l$ by $l-2$ in 3rd term
\begin{align*}
     E_{1}&= \sum_{l=1}^{k-1}\sum_{q \geq \frac{k-l+1}{2}}^{k-l+1}\frac{(k+l-1)!(-1)^{k+l+q}2^{2q}(k-1)!(k-l+1-q)}{(l-1)!2^{l-1}(2 q-k+l-1)!(k-l+1-q)!(k-q)!}Q^{k-1-q}A^{2 q-k+l}B^{k-l+1}\\
     &+\sum_{l=2}^{k+2}\sum_{q \geq \frac{k-l+2}{2}}^{k-l+2} \frac{(k+l-2)!(-1)^{k+l+q}2^{2q-l}(k-1)!(k-l+2-q)}{(l-2)!(2 q-k+l-2)!(k-l+2-q)!(k-q)!}Q^{k-1-q}A^{2 q-k+l}B^{k-l+2}\\
     &+\sum_{l=2}^{k+2}\sum_{q \ge \frac{k-l}{2}}^{k-l+1}\frac{(k+l-2)!(-1)^{k+l+q-1}2^{2q-l+2}(k-1)!(k-l-q+1)}{(l-2)!(2q-k+l)!\,(k-l-q+1)!(k-q-1)!}
Q^{k-q-1}A^{2q-k+l}B^{k-l+2}.
\end{align*}
Note that first term of $E_{1}$ vanishes for $q=k-l+1$ and lower limit of $q$ can be lowers down to $\frac{k-l}{2}$ by multilpying it by $2q-k+l$. Also note that 2nd term vanishes for $q=k-l+2$ and lower limit can be lowers down. Hence we get
\begin{align*}
     E_{1}&= \sum_{l=1}^{k-1}\sum_{q \geq \frac{k-l}{2}}^{k-l}\frac{(k+l-1)!(-1)^{k+l+q}2^{2q-l+1}(k-1)!(2q-k+l)}{(l-1)!(2 q-k+l)!(k-l-q)!(k-q)!}Q^{k-1-q}A^{2 q-k+l}B^{k-l+1}\\
     &+\sum_{l=2}^{k+2}\sum_{q \geq \frac{k-l}{2}}^{k-l+1} \frac{(k+l-2)!(-1)^{k+l+q}2^{2q-l}(k-1)!(2q-k+l-1)(2q-k+l)}{(l-2)!(2 q-k+l)!(k-l-q+1)!(k-q)!}Q^{k-1-q}A^{2 q-k+l}B^{k-l+2}\\
     &+\sum_{l=2}^{k+2}\sum_{q \ge \frac{k-l}{2}}^{k-l+1}\frac{(k+l-2)!(-1)^{k+l+q-1}2^{2q-l+2}(k-1)!(k-l-q+1)}{(l-2)!(2q-k+l)!(k-l-q+1)!(k-q-1)!}
Q^{k-q-1}A^{2q-k+l}B^{k-l+2}.
\end{align*}
    Adding $E_{1}$ and $E_{2}$, we have
    \begin{align*}
    & E_{1}+E_{2}\\=& \sum_{l=2}^{k+2}\sum_{q \geq \frac{k-l}{2}}^{k-l+1}\frac{(k+l-2)!(-1)^{k+l+q}2^{2q-l}(k-1)!\left[(2q-k+l-1)(2q-k+l)-
4(k-l-q+1)(k-q)\right]}{(l-2)!(2 q-k+l)!(k-l-q+1)!(k-q)!} \\
&\times Q^{k-q-1}A^{2q-k+l}B^{k-l+2}\\
=& -\sum_{l=2}^{k+2}\sum_{q \geq \frac{k-l}{2}}^{k-l+1}\frac{(k+l-2)!(-1)^{k+l+q}2^{2q-l}(k-1)!}{(l-2)!(2 q-k+l)!(k-l-q+1)!(k-q)!} \left[(3k+l-2q-2kl-4kq+3k^2-l^2)\right]\\
&\times Q^{k-q-1}A^{2q-k+l}B^{k-l+2}.
\end{align*}
Now consider $E_{1}+E_{2}+E_{3}$, we get
\begin{multline*}
    E_{1}+E_{2}+E_{3}
= -\sum_{l=k+1}^{k+2}\sum_{q \geq \frac{k-l}{2}}^{k-l+1}\frac{(k+l-2)!(-1)^{k+l+q}2^{2q-l}(k-1)!\left[(3k+l-2q-2kl-4kq+3k^2-l^2)\right]}{(l-2)!(2 q-k+l)!(k-l-q+1)!(k-q)!}\\ \times Q^{k-q-1}A^{2q-k+l}B^{k-l+2}.
\end{multline*}
It is easy to check that $(3k+l-2q-2kl-4kq+3k^2-l^2)=0$ for $(l,q)=(k+1,0)$ and $(l,q)=(k+2,q=-1)$, this gives $E_{1}+E_{2}+E_{3}\equiv 0.$ Hence $F(t,u)\equiv0.$
\end{proof}
\section{Acknowledgments}
PC acknowledges the support of the Department of Atomic Energy,  Government of India, for PhD Fellowship.

VPK would like to thank the Isaac Newton Institute for Mathematical Sciences, Cambridge, UK, for support and hospitality during \emph{Rich and Nonlinear Tomography - a multidisciplinary approach} in 2023 (supported by EPSRC Grant Number EP/R014604/1).
Additionally, VPK acknowledges the support of the Department of Atomic Energy,  Government of India, under
Project No.  12-R\&D-TFR-5.01-0520.

AT acknowledges the supported by the Anusandhan National Research Foundation (ANRF), under the scheme National
Post-Doctoral Fellowship, file no. PDF/2025/005843.
\bibliographystyle{plain}
\bibliography{references}

@article{chatterjee2026unifiedrangecharacterizationspherical,
      title={A Unified Range Characterization for the Spherical mean transform}, 
      author={Pradipta Chatterjee and Nisha Singhal and Abhilash Tushir},
      year={2026},
      Journal={arXiv.2605.27034},
      archivePrefix={arXiv},
      primaryClass={math.AP},
      url={https://arxiv.org/abs/2605.27034}, 
}

@book {Atk:book,
    AUTHOR = {Atkinson, Kendall and Han, Weimin},
     TITLE = {Spherical harmonics and approximations on the unit sphere: an
              introduction},
    SERIES = {Lecture Notes in Mathematics},
    VOLUME = {2044},
 PUBLISHER = {Springer, Heidelberg},
      YEAR = {2012},
     PAGES = {x+244},
      ISBN = {978-3-642-25982-1},
   MRCLASS = {41-02 (33C55 41A30 41A63 42A10)},
  MRNUMBER = {2934227},
MRREVIEWER = {Feng\ Dai},
       DOI = {10.1007/978-3-642-25983-8},
       URL = {https://doi.org/10.1007/978-3-642-25983-8},
}

@article {whitney43,
    AUTHOR = {Whitney, Hassler},
     TITLE = {Differentiable even functions},
   JOURNAL = {Duke Math. J.},
  FJOURNAL = {Duke Mathematical Journal},
    VOLUME = {10},
      YEAR = {1943},
     PAGES = {159--160},
      ISSN = {0012-7094,1547-7398},
   MRCLASS = {27.0X},
  MRNUMBER = {7783},
MRREVIEWER = {A.\ E.\ Taylor},
       URL = {http://projecteuclid.org/euclid.dmj/1077471799},
}

@article{Finch_2006,
doi = {10.1088/0266-5611/22/3/012},
url = {https://doi.org/10.1088/0266-5611/22/3/012},
year = {2006},
month = {may},
publisher = {},
volume = {22},
number = {3},
pages = {923},
author = {Finch, David and Rakesh},
title = {The range of the spherical mean value operator for functions supported in a ball},
journal = {Inverse Problems}
}

@book{Gabor,
    AUTHOR = {Szeg\H o, G\'abor},
     TITLE = {Orthogonal polynomials},
    SERIES = {American Mathematical Society Colloquium Publications},
    VOLUME = {Vol. XXIII},
   EDITION = {Fourth},
 PUBLISHER = {American Mathematical Society, Providence, RI},
      YEAR = {1975},
     PAGES = {xiii+432},
   MRCLASS = {42A52 (33A65)},
  MRNUMBER = {372517},
}

@article {AQ2,
    AUTHOR = {Agranovsky, Mark L. and Quinto, Eric Todd},
     TITLE = {Stationary sets for the wave equation in crystallographic
              domains},
   JOURNAL = {Trans. Amer. Math. Soc.},
  FJOURNAL = {Transactions of the American Mathematical Society},
    VOLUME = {355},
      YEAR = {2003},
    NUMBER = {6},
     PAGES = {2439--2451},
      ISSN = {0002-9947,1088-6850},
   MRCLASS = {35L05 (35B05 35L20 35S30)},
  MRNUMBER = {1973997},
MRREVIEWER = {R.\ G.\ Airapetyan},
       DOI = {10.1090/S0002-9947-03-03228-8},
       URL = {https://doi.org/10.1090/S0002-9947-03-03228-8},
}

@article {AQ3,
    AUTHOR = {Agranovsky, Mark L. and Quinto, Eric Todd},
     TITLE = {Geometry of stationary sets for the wave equation in {$\mathbb{
              R^n}$}: the case of finitely supported initial data},
   JOURNAL = {Duke Math. J.},
  FJOURNAL = {Duke Mathematical Journal},
    VOLUME = {107},
      YEAR = {2001},
    NUMBER = {1},
     PAGES = {57--84},
      ISSN = {0012-7094,1547-7398},
   MRCLASS = {35L05 (35L10 35S30 44A12)},
  MRNUMBER = {1815250},
MRREVIEWER = {Victor\ V.\ Zharinov},
       DOI = {10.1215/S0012-7094-01-10714-X},
       URL = {https://doi.org/10.1215/S0012-7094-01-10714-X},
}

@article{Linh,
author={Linh V. Nguyen},
	title = {Range description for a spherical mean transform on spaces of constant curvature},
	journal = {JJAMA},
	volume = {128},
	number = {},
	pages = {191–214},
	year = {2016},
	
}

@article {AN,
    AUTHOR = {Agranovsky, Mark and Nguyen, Linh V.},
     TITLE = {Range conditions for a spherical mean transform and global
              extendibility of solutions of the {D}arboux equation},
   JOURNAL = {J. Anal. Math.},
  FJOURNAL = {Journal d'Analyse Math\'{e}matique},
    VOLUME = {112},
      YEAR = {2010},
     PAGES = {351--367},
      ISSN = {0021-7670},
   MRCLASS = {35L20 (35A01 44A15)},
  MRNUMBER = {2763005},
MRREVIEWER = {Chong Kyu Han},
       DOI = {10.1007/s11854-010-0033-0},
       URL = {https://doi.org/10.1007/s11854-010-0033-0},
}

@article {Clark1999,
    AUTHOR = {Clarkson, Eric},
     TITLE = {Projections onto the range of the exponential {R}adon
              transform and reconstruction algorithms},
   JOURNAL = {Inverse Problems},
  FJOURNAL = {Inverse Problems. An International Journal on the Theory and
              Practice of Inverse Problems, Inverse Methods and Computerized
              Inversion of Data},
    VOLUME = {15},
      YEAR = {1999},
    NUMBER = {2},
     PAGES = {563--571},
      ISSN = {0266-5611,1361-6420},
   MRCLASS = {44A12},
  MRNUMBER = {1684475},
MRREVIEWER = {Aleksander\ Denisiuk},
       DOI = {10.1088/0266-5611/15/2/014},
       URL = {https://doi.org/10.1088/0266-5611/15/2/014},
}

@article {DK,
    AUTHOR = {Do, N. and Kunyansky, L.},
     TITLE = {Theoretically exact photoacoustic reconstruction from
              spatially and temporally reduced data},
   JOURNAL = {Inverse Problems},
  FJOURNAL = {Inverse Problems. An International Journal on the Theory and
              Practice of Inverse Problems, Inverse Methods and Computerized
              Inversion of Data},
    VOLUME = {34},
      YEAR = {2018},
    NUMBER = {9},
     PAGES = {094004, 26},
      ISSN = {0266-5611,1361-6420},
   MRCLASS = {65M32 (35L15 44A12)},
  MRNUMBER = {3830143},
MRREVIEWER = {Natesan\ Barani Balan},
       DOI = {10.1088/1361-6420/aacfac},
       URL = {https://doi.org/10.1088/1361-6420/aacfac},
}

@article {Kuchment-Kunyansky-TAT,
    AUTHOR = {Kuchment, Peter and Kunyansky, Leonid},
     TITLE = {Mathematics of thermoacoustic tomography},
   JOURNAL = {European J. Appl. Math.},
  FJOURNAL = {European Journal of Applied Mathematics},
    VOLUME = {19},
      YEAR = {2008},
    NUMBER = {2},
     PAGES = {191--224},
      ISSN = {0956-7925,1469-4425},
   MRCLASS = {92C55 (35L20 44A12 76Q05 78A70 80A99)},
  MRNUMBER = {2400720},
       DOI = {10.1017/S0956792508007353},
       URL = {https://doi.org/10.1017/S0956792508007353},
}

@incollection {Nat83,
    AUTHOR = {Natterer, Frank},
     TITLE = {Exploiting the ranges of {R}adon transforms in tomography},
 BOOKTITLE = {Numerical treatment of inverse problems in differential and
              integral equations ({H}eidelberg, 1982)},
    SERIES = {Wissenschaft und Kultur},
     PAGES = {290--303},
 PUBLISHER = {Birkh\"auser, Boston, MA},
      YEAR = {1983},
      ISBN = {3-7643-3125-9},
   MRCLASS = {44A15 (53C65 92A07)},
  MRNUMBER = {714577},
}

@article{SKPatch_2004,
doi = {10.1088/0031-9155/49/11/013},
url = {https://doi.org/10.1088/0031-9155/49/11/013},
year = {2004},
month = {may},
publisher = {},
volume = {49},
number = {11},
pages = {2305},
author = {S K Patch},
title = {Thermoacoustic tomography—consistency conditions and the partial scan problem},
journal = {Physics in Medicine \& Biology}
}

@article{Anastasio2001FBP,
  author    = {Mark A. Anastasio and Xiaochuan Pan and Eric Clarkson},
  title     = {Comments on the Filtered Backprojection Algorithm, Range Conditions, and the Pseudoinverse Solution},
  journal   = {IEEE Transactions on Medical Imaging},
  volume    = {20},
  pages     = {539--542},
  year      = {2001},
  doi       = {10.1109/42.929620},
  issn      = {0278-0062},
  publisher = {IEEE}
}

@article{AAKN1,
title = {A simple range characterization for spherical mean transform in odd dimensions and its applications},
journal = {Inverse Problems and Imaging},
pages = {},
year = {2026},
issn = {1930-8337},
doi = {10.3934/ipi.2026047},
url = {https://www.aimsciences.org/article/id/6a672fd4772a870addfe142c},
author = {Divyansh Agrawal and Gaik Ambartsoumian and Venkateswaran P. Krishnan and Nisha Singhal}
}

@article {R,
    AUTHOR = {Rubin, Boris},
     TITLE = {Inversion formulae for the spherical mean in odd dimensions
              and the {E}uler-{P}oisson-{D}arboux equation},
   JOURNAL = {Inverse Problems},
  FJOURNAL = {Inverse Problems. An International Journal on the Theory and
              Practice of Inverse Problems, Inverse Methods and Computerized
              Inversion of Data},
    VOLUME = {24},
      YEAR = {2008},
    NUMBER = {2},
     PAGES = {025021, 10},
      ISSN = {0266-5611},
   MRCLASS = {44A12 (26A33 35Q15 35R30)},
  MRNUMBER = {2408558},
MRREVIEWER = {Dmitry G. Shepelsky},
       DOI = {10.1088/0266-5611/24/2/025021},
       URL = {https://doi.org/10.1088/0266-5611/24/2/025021},
}

@article {K,
    AUTHOR = {Kunyansky, Leonid A.},
     TITLE = {Explicit inversion formulae for the spherical mean {R}adon
              transform},
   JOURNAL = {Inverse Problems},
  FJOURNAL = {Inverse Problems. An International Journal on the Theory and
              Practice of Inverse Problems, Inverse Methods and Computerized
              Inversion of Data},
    VOLUME = {23},
      YEAR = {2007},
    NUMBER = {1},
     PAGES = {373--383},
      ISSN = {0266-5611},
   MRCLASS = {44A12 (65R10)},
  MRNUMBER = {2302980},
MRREVIEWER = {Fritz Keinert},
       DOI = {10.1088/0266-5611/23/1/021},
       URL = {https://doi.org/10.1088/0266-5611/23/1/021},
}

@article{ref:AmbKuch,
author = {Gaik Ambartsoumian and Peter Kuchment},
title = {{On the injectivity of the circular Radon transform}},
journal =  {Inverse Problems},
volume = 21,
pages = {473--485}, year = 2005}

@article{ref:AmbKuch-range,
author = {Gaik Ambartsoumian and Peter Kuchment},
title = {{A range description for the planar circular Radon
transform}},
journal =  {SIAM J. Math. Anal.},
volume = 38,
number = 2,
pages = {681--692},
year = 2006}

@article{And,
key = {And},
author = {Lars-Erik Andersson},
title = {{On the determination of a function from
spherical averages}},
journal = {SIAM J.  Math.  Anal.},
volume =  19,
year = 1988,
pages = {214-232},
}

@article{KK,
author={Kuchment, Peter and Kunyansky, Leonid},
TITLE={Range description for the free space wave operator
and the spherical means transform},
NOTE={In preparation},
YEAR={2023},
}

@article {AAKN2,
    AUTHOR = {Agrawal, Divyansh and Ambartsoumian, Gaik and Krishnan,
              Venkateswaran P. and Singhal, Nisha},
     TITLE = {On the null space of the backprojection operator and {R}ubin's
              conjecture for the spherical mean transform},
   JOURNAL = {Inverse Problems},
  FJOURNAL = {Inverse Problems. An International Journal on the Theory and
              Practice of Inverse Problems, Inverse Methods and Computerized
              Inversion of Data},
    VOLUME = {40},
      YEAR = {2024},
    NUMBER = {12},
     PAGES = {Paper No. 125018, 15},
      ISSN = {0266-5611,1361-6420},
   MRCLASS = {44A12 (26A33)},
  MRNUMBER = {4844592},
       DOI = {10.1088/1361-6420/ad8fc8},
       URL = {https://doi.org/10.1088/1361-6420/ad8fc8},
}

@article{Finch-P-R,
  author =   {David Finch and Sarah Patch and Rakesh},
  title =    {Determining a function from its mean values over
                  a family of spheres},
year =   {2004},
journal = {SIAM J. Math. Anal.},
volume =  {35},
pages = {1213--1240},
}

@book {Palamodov2016,
    AUTHOR = {Palamodov, Victor},
     TITLE = {Reconstruction from integral data},
    SERIES = {Monographs and Research Notes in Mathematics},
 PUBLISHER = {CRC Press, Boca Raton, FL},
      YEAR = {2016},
     PAGES = {xi+169},
      ISBN = {978-1-4987-1010-7},
   MRCLASS = {44A12 (82D25 92C55)},
  MRNUMBER = {3618153},
MRREVIEWER = {Aleksander\ Denisiuk},
       DOI = {10.1201/b19575},
       URL = {https://doi.org/10.1201/b19575},
}

@incollection {Agronovasky96,
    AUTHOR = {Agranovsky, Mark L. and Quinto, Eric Todd},
     TITLE = {Injectivity of the spherical mean operator and related
              problems},
 BOOKTITLE = {Complex analysis, harmonic analysis and applications
              ({B}ordeaux, 1995)},
    SERIES = {Pitman Res. Notes Math. Ser.},
    VOLUME = {347},
     PAGES = {12--36},
 PUBLISHER = {Longman, Harlow},
      YEAR = {1996},
      ISBN = {0-582-28698-0},
   MRCLASS = {44A12},
  MRNUMBER = {1402020},
MRREVIEWER = {\'Arp\'ad\ Kurusa},
}

@article {Denisjuk1999,
    AUTHOR = {Denisjuk, Alexander},
     TITLE = {Integral geometry on the family of semi-spheres},
   JOURNAL = {Fract. Calc. Appl. Anal.},
  FJOURNAL = {Fractional Calculus \& Applied Analysis. An International
              Journal for Theory and Applications},
    VOLUME = {2},
      YEAR = {1999},
    NUMBER = {1},
     PAGES = {31--46},
      ISSN = {1311-0454},
   MRCLASS = {53C65},
  MRNUMBER = {1679627},
}

@article {Agranovsky:1999,
    AUTHOR = {Agranovsky, M. L. and Volchkov, V. V. and Zalcman, L. A.},
     TITLE = {Conical uniqueness sets for the spherical {R}adon transform},
   JOURNAL = {Bull. London Math. Soc.},
  FJOURNAL = {The Bulletin of the London Mathematical Society},
    VOLUME = {31},
      YEAR = {1999},
    NUMBER = {2},
     PAGES = {231--236},
      ISSN = {0024-6093,1469-2120},
   MRCLASS = {44A12},
  MRNUMBER = {1664137},
       DOI = {10.1112/S0024609398005396},
       URL = {https://doi.org/10.1112/S0024609398005396},
}

@Article{Ambartsoumian2018,
  author   = {Ambartsoumian, Gaik and Gouia-Zarrad, Rim and Krishnan, Venkateswaran P. and Roy, Souvik},
  title    = {Image reconstruction from radially incomplete spherical {R}adon data},
  journal  = {European J. Appl. Math.},
  year     = {2018},
  volume   = {29},
  number   = {3},
  pages    = {470--493},
  issn     = {0956-7925},
  doi      = {10.1017/S0956792517000250},
  mrnumber = {3788452},
}

@InCollection{Ambartsoumian2015,
  author    = {Ambartsoumian, Gaik and Krishnan, Venkateswaran P.},
  title     = {Inversion of a class of circular and elliptical {R}adon transforms},
  booktitle = {Complex analysis and dynamical systems {VI}. {P}art 1},
  publisher = {Amer. Math. Soc., Providence, RI},
  year      = {2015},
  volume    = {653},
  series    = {Contemp. Math.},
  pages     = {1--12},
  doi       = {10.1090/conm/653/13174},
  mrnumber  = {3453064},
}

@book {CH_Book,
    AUTHOR = {Courant, Richard and Hilbert, David},
     TITLE = {Methods of Mathematical Physics. {V}ol. {II}},
    SERIES = {Wiley Classics Library},
      NOTE = {Partial differential equations,
              Reprint of the 1962 original,
              A Wiley-Interscience Publication},
 PUBLISHER = {John Wiley \& Sons, Inc., New York},
      YEAR = {1989},
     PAGES = {xxii+830},
      ISBN = {0-471-50439-4},
   MRCLASS = {35-00 (00A05 01A75)},
  MRNUMBER = {1013360},
}

@book {Natterer_book,
    AUTHOR = {Natterer, Frank},
     TITLE = {The Mathematics of Computerized Tomography},
    SERIES = {Classics in Applied Mathematics},
    VOLUME = {32},
      NOTE = {Reprint of the 1986 original},
 PUBLISHER = {Society for Industrial and Applied Mathematics (SIAM),
              Philadelphia, PA},
      YEAR = {2001},
     PAGES = {xviii+222},
      ISBN = {0-89871-493-1},
   MRCLASS = {00A69 (44A12 65R10 68U99 92C55)},
  MRNUMBER = {1847845},
MRREVIEWER = {Fritz Keinert},
       DOI = {10.1137/1.9780898719284},
       URL = {https://doi.org/10.1137/1.9780898719284},
}

@article {Agranovsky-Kuchment-Quinto,
    AUTHOR = {Agranovsky, Mark and Kuchment, Peter and Quinto, Eric Todd},
     TITLE = {Range descriptions for the spherical mean {R}adon transform},
   JOURNAL = {J. Funct. Anal.},
  FJOURNAL = {Journal of Functional Analysis},
    VOLUME = {248},
      YEAR = {2007},
    NUMBER = {2},
     PAGES = {344--386},
      ISSN = {0022-1236},
   MRCLASS = {47G10 (43A85 44A12 94A12)},
  MRNUMBER = {2335579},
MRREVIEWER = {Keisaku Kumahara},
       DOI = {10.1016/j.jfa.2007.03.022},
       URL = {https://doi.org/10.1016/j.jfa.2007.03.022},
}

@article {Finch-Haltmeir-Rakesh_even-inversion,
    AUTHOR = {Finch, David and Haltmeier, Markus and Rakesh},
     TITLE = {Inversion of spherical means and the wave equation in even
              dimensions},
   JOURNAL = {SIAM J. Appl. Math.},
  FJOURNAL = {SIAM Journal on Applied Mathematics},
    VOLUME = {68},
      YEAR = {2007},
    NUMBER = {2},
     PAGES = {392--412},
      ISSN = {0036-1399,1095-712X},
   MRCLASS = {35R30 (35L05 65R32 92C55)},
  MRNUMBER = {2366991},
MRREVIEWER = {Alexey\ V.\ Borovskikh},
       DOI = {10.1137/070682137},
       URL = {https://doi.org/10.1137/070682137},
}

@article {Ambartsoumian-Zarrad-Lewis,
    AUTHOR = {Ambartsoumian, Gaik and Gouia-Zarrad, Rim and Lewis, Matthew
              A.},
     TITLE = {Inversion of the circular {R}adon transform on an annulus},
   JOURNAL = {Inverse Problems},
  FJOURNAL = {Inverse Problems. An International Journal on the Theory and
              Practice of Inverse Problems, Inverse Methods and Computerized
              Inversion of Data},
    VOLUME = {26},
      YEAR = {2010},
    NUMBER = {10},
     PAGES = {105015, 11},
      ISSN = {0266-5611,1361-6420},
   MRCLASS = {44A12 (65R99 92C55)},
  MRNUMBER = {2719776},
MRREVIEWER = {H.\ S. P. Shrivastava},
       DOI = {10.1088/0266-5611/26/10/105015},
       URL = {https://doi.org/10.1088/0266-5611/26/10/105015},
}

@article{Salman:2014,
title = {An inversion formula for the spherical mean transform with data on an ellipsoid in two and three dimensions},
journal = {Journal of Mathematical Analysis and Applications},
volume = {420},
number = {1},
pages = {612-620},
year = {2014},
issn = {0022-247X},
doi = {https://doi.org/10.1016/j.jmaa.2014.05.007},
url = {https://www.sciencedirect.com/science/article/pii/S0022247X14004429},
author = {Yehonatan Salman}
}

@article {Salman,
    AUTHOR = {Salman, Yehonatan},
     TITLE = {Recovering functions from the spherical mean transform with
              limited radii data by expansion into spherical harmonics},
   JOURNAL = {J. Math. Anal. Appl.},
  FJOURNAL = {Journal of Mathematical Analysis and Applications},
    VOLUME = {465},
      YEAR = {2018},
    NUMBER = {1},
     PAGES = {331--347},
      ISSN = {0022-247X,1096-0813},
   MRCLASS = {33C55},
  MRNUMBER = {3806707},
MRREVIEWER = {Tetiana\ A.\ Stepanyuk},
       DOI = {10.1016/j.jmaa.2018.05.019},
       URL = {https://doi.org/10.1016/j.jmaa.2018.05.019},
}

@article {Seeley,
    AUTHOR = {Seeley, R. T.},
     TITLE = {Spherical harmonics},
   JOURNAL = {Amer. Math. Monthly},
  FJOURNAL = {American Mathematical Monthly},
    VOLUME = {73},
      YEAR = {1966},
    NUMBER = {4, part II},
     PAGES = {115--121},
      ISSN = {0002-9890,1930-0972},
   MRCLASS = {33.27},
  MRNUMBER = {201695},
MRREVIEWER = {Ram\ Kishore\ Saxena},
       DOI = {10.2307/2313760},
       URL = {https://doi.org/10.2307/2313760},
}

@article {Eknaraynan2007,
    AUTHOR = {Narayanan, E. K. and Rawat, R. and Ray, S. K.},
     TITLE = {Approximation by {$K$}-finite functions in {$L^p$} spaces},
   JOURNAL = {Israel J. Math.},
  FJOURNAL = {Israel Journal of Mathematics},
    VOLUME = {161},
      YEAR = {2007},
     PAGES = {187--207},
      ISSN = {0021-2172,1565-8511},
   MRCLASS = {43A85 (46E30)},
  MRNUMBER = {2350162},
MRREVIEWER = {Mark\ Agranovsky},
       DOI = {10.1007/s11856-007-0078-7},
       URL = {https://doi.org/10.1007/s11856-007-0078-7},
}

@article {EKnarayanan2009,
    AUTHOR = {Narayanan, E. K. and Rawat, R. and Ray, S. K.},
     TITLE = {Approximation by {$K$}-finite functions on {$L^p$} spaces},
   JOURNAL = {J. Anal.},
  FJOURNAL = {The Journal of Analysis},
    VOLUME = {17},
      YEAR = {2009},
     PAGES = {61--66},
      ISSN = {0971-3611,2367-2501},
   MRCLASS = {43A85 (46E30)},
  MRNUMBER = {2722602},
MRREVIEWER = {Radu\ Miculescu},
}

@article {Rama2009,
    AUTHOR = {Rawat, Rama and Srivastava, R. K.},
     TITLE = {Twisted spherical means in annular regions in {$\Bbb C^n$} and
              support theorems},
   JOURNAL = {Ann. Inst. Fourier (Grenoble)},
  FJOURNAL = {Universit\'e{} de Grenoble. Annales de l'Institut Fourier},
    VOLUME = {59},
      YEAR = {2009},
    NUMBER = {6},
     PAGES = {2509--2523},
      ISSN = {0373-0956,1777-5310},
   MRCLASS = {32A50 (43A85 44A35)},
  MRNUMBER = {2640928},
MRREVIEWER = {E.\ K.\ Narayanan},
       DOI = {10.5802/aif.2498},
       URL = {https://doi.org/10.5802/aif.2498},
}

@article {Rama2011,
    AUTHOR = {Rawat, Rama and Srivastava, R. K.},
     TITLE = {Spherical means in annular regions in the {$n$}-dimensional
              real hyperbolic spaces},
   JOURNAL = {Proc. Indian Acad. Sci. Math. Sci.},
  FJOURNAL = {Indian Academy of Sciences. Proceedings. Mathematical
              Sciences},
    VOLUME = {121},
      YEAR = {2011},
    NUMBER = {3},
     PAGES = {311--325},
      ISSN = {0253-4142,0973-7685},
   MRCLASS = {42C15 (26E60 43A90)},
  MRNUMBER = {2867985},
MRREVIEWER = {W.\ \.Zelazko},
       DOI = {10.1007/s12044-011-0037-4},
       URL = {https://doi.org/10.1007/s12044-011-0037-4},
}

@book {John-book,
    AUTHOR = {John, Fritz},
     TITLE = {Plane Waves and Spherical Means Applied to Partial
              Differential Equations},
      NOTE = {Reprint of the 1955 original},
 PUBLISHER = {Dover Publications, Inc., Mineola, NY},
      YEAR = {2004},
     PAGES = {iv+172},
      ISBN = {0-486-43804-X},
   MRCLASS = {35-02},
  MRNUMBER = {2098409},
}

@article{CKT,
    AUTHOR = {Chatterjee, Pradipta and Krishnan, Venkateswaran P. and
              Tushir, Abhilash},
     TITLE = {Inversion of spherical mean transform in odd dimensions with
              partial data via system of {ODE}s},
   JOURNAL = {Anal. Math. Phys.},
  FJOURNAL = {Analysis and Mathematical Physics},
    VOLUME = {16},
      YEAR = {2026},
    NUMBER = {5},
     PAGES = {Paper No. 123},
      ISSN = {1664-2368,1664-235X},
   MRCLASS = {33C55 (35R30 44A12 44A15 44A20)},
  MRNUMBER = {5128561},
       DOI = {10.1007/s13324-026-01267-7},
       URL = {https://doi.org/10.1007/s13324-026-01267-7},
}

@article {Rhee,
    AUTHOR = {Rhee, Haewun},
     TITLE = {A representation of the solutions of the {D}arboux equation in
              odd-dimensional spaces},
   JOURNAL = {Trans. Amer. Math. Soc.},
  FJOURNAL = {Transactions of the American Mathematical Society},
    VOLUME = {150},
      YEAR = {1970},
     PAGES = {491--498},
      ISSN = {0002-9947,1088-6850},
   MRCLASS = {35.06},
  MRNUMBER = {262647},
MRREVIEWER = {E.\ J.\ Scott},
       DOI = {10.2307/1995531},
       URL = {https://doi.org/10.2307/1995531},
}

@article {Agranovsky-Finch-Kuchment-range,
    AUTHOR = {Agranovsky, Mark and Finch, David and Kuchment, Peter},
     TITLE = {Range conditions for a spherical mean transform},
   JOURNAL = {Inverse Probl. Imaging},
  FJOURNAL = {Inverse Problems and Imaging},
    VOLUME = {3},
      YEAR = {2009},
    NUMBER = {3},
     PAGES = {373--382},
      ISSN = {1930-8337,1930-8345},
   MRCLASS = {44A12 (92C55)},
  MRNUMBER = {2557910},
MRREVIEWER = {Aleksander\ Denisiuk},
       DOI = {10.3934/ipi.2009.3.373},
       URL = {https://doi.org/10.3934/ipi.2009.3.373},
}

@article{ambartsoumian2014exterior,
  title={Exterior/interior problem for the circular means transform with applications to intravascular imaging},
  author={Ambartsoumian, Gaik and Kunyansky, Leonid},
  journal={Inverse Problems and Imaging},
  volume={8},
  number={2},
  pages={339--359},
  year={2014},
  publisher={Inverse Problems and Imaging}
}

@article{ambartsoumian2015inversion,
  title={Inversion of a class of circular and elliptical {R}adon transforms},
  author={Ambartsoumian, Gaik and Krishnan, Venkateswaran P.},
  journal={Contemporary Mathematics},
  volume={653},
  year={2015},
  pages={1--12}
}

@article{nguyen2009family,
  title={A family of inversion formulas in thermoacoustic tomography},
  author={Nguyen, Linh V.},
  journal={Inverse Problems and Imaging},
  volume={3},
  number={4},
  pages={649--675},
  year={2009},
  publisher={Inverse Problems and Imaging}
}

@article{norton1980reconstruction,
  title={Reconstruction of a two-dimensional reflecting medium over a circular domain: Exact solution},
  author={Norton, Stephen J.},
  journal={The Journal of the Acoustical Society of America},
  volume={67},
  number={4},
  pages={1266--1273},
  year={1980},
  publisher={Acoustical Society of America}
}

@article{aramyan2020recovering,
  title={To recovering the moments from the spherical mean {R}adon transform},
  author={Aramyan, Rafik H. and Mnatsakanov, Robert M.},
  journal={Journal of Mathematical Analysis and Applications},
  volume={490},
  number={2},
  pages={124334},
  year={2020},
  publisher={Elsevier}
}

@article{agranovsky1996approximation,
  title={Approximation by spherical waves in ${L}^p$-spaces},
  author={Agranovsky, Mark and Berenstein, Carlos and Kuchment, Peter},
  journal={J. Geom. Anal.},
  volume={6},
  number={3},
  pages={365--383},
  year={1996},
  publisher={Springer}
}

@article{agranovsky1996injectivity,
  title={Injectivity sets for the {R}adon transform over circles and complete systems of radial functions},
  author={Agranovsky, Mark and Quinto, Eric Todd},
  journal={J. Funct. Anal.},
  volume={139},
  number={2},
  pages={383--414},
  year={1996},
  publisher={Elsevier}
}

@article{Kruger:1995,
author = {Kruger, Robert A. and Liu, Pingyu and Fang, Yuncai and Appledorn, C. Robert},
title = {Photoacoustic ultrasound {(PAUS)}—Reconstruction tomography},
journal = {Medical Physics},
volume = {22},
number = {10},
pages = {1605-1609},
doi = {https://doi.org/10.1118/1.597429},
eprint = {https://aapm.onlinelibrary.wiley.com/doi/pdf/10.1118/1.597429},
year = {1995}
}
  \end{document}